%% file: main.tex
\documentclass[a4paper,12pt]{article}

\usepackage{cmap}     
\usepackage{amsmath,amssymb}
\usepackage[utf8]{inputenc}
\usepackage[T2A]{fontenc}  
\usepackage{graphicx}    
\usepackage[margin=1in]{geometry}
\usepackage{fancyhdr}    
\usepackage{setspace}
\usepackage{wrapfig}
\usepackage{subfig}
\usepackage{listings}
\usepackage{color}
\usepackage{setspace}
\usepackage{textcomp}
\usepackage{float}
\usepackage{listings}
\usepackage{amsthm}
\usepackage{algorithm}
\usepackage{algpseudocode}
\usepackage{afterpage}
\usepackage[toc,page]{appendix}
\usepackage{natbib}
\usepackage{bbm}
\usepackage{footnote}

\makesavenoteenv{tabular}

\newtheorem{theorem}{Theorem}
\newtheorem{lemma}{Lemma}
\newtheorem{corollary}{Corollary}

\newtheorem{assumption}{Assumption}

\newcommand{\texp}{t_*}
\newcommand{\dtilde}[1]{\breve{#1}}

\let\given\givenbase

\DeclareMathOperator{\cov}{cov}

\DeclareMathOperator{\gammadistr}{Gamma}
\DeclareMathOperator{\argmax}{argmax}
\begin{document}

\title{Simulation of particle systems interacting through hitting times} 
\author{Vadim Kaushansky\thanks{The first author gratefully acknowledges support from the Economic and Social Research Council and Bank of America Merrill Lynch.}
\footnote{\emph{Corresponding author}, Mathematical Institute \& Oxford-Man Institute, University of Oxford, Andrew Wiles Building, Woodstock Rd, Oxford, OX2 6GG, UK, E-mail: 
vadim.kaushansky@maths.ox.ac.uk}, Christoph Reisinger\footnote{Mathematical Institute  \& Oxford-Man Institute, University of Oxford, Andrew Wiles Building, Woodstock Rd, Oxford, OX2 6GG, UK, E-mail: christoph.reisinger@maths.ox.ac.uk}}    
\date{}
   
\maketitle 
\begin{abstract}
We develop an Euler-type particle method for the simulation of a McKean--Vlasov equation arising from a mean-field model with positive feedback from hitting a boundary. Under  assumptions on the parameters which ensure differentiable solutions, we establish convergence of order $1/2$ in the time step. Moreover, we give a modification of the scheme using Brownian bridges and local mesh refinement, which improves the 
order to $1$. We confirm our theoretical results with numerical tests and empirically investigate cases with blow-up.
\end{abstract}

\medskip

\input{Introduction.tex}

\input{Auxiliary.tex}

\input{Uniform_convergence.tex}
\input{Monte_Carlo_conv.tex}
\input{Brownian_bridge.tex}
\input{Numerical_results.tex}

\input{Conclusion.tex}
\section*{Acknowledgements}
We thank Alex Lipton, Andreas Søjmark, and Sean Ledger for useful comments and interesting discussion. We assume full responsibility for any remaining mistakes.

\input{Appendix.tex}

\bibliographystyle{apalike}
\bibliography{main}

\end{document}

%% file: Introduction.tex
\section{Introduction}

There has been a recent surge in interest in mean-field problems, both from a theoretical and applications perspective.
We focus on models where the interaction derives from feedback on the system when a certain threshold is hit.
Application areas include electrical surges in networks of neurons and systemic risk in financial markets.
As analytic solutions are generally not known, numerical methods are inevitable, but still lacking.

We therefore propose and analyse numerical schemes for the simulation of a specific McKean--Vlasov equation which exhibits
key features of these models, namely
\begin{align}
	Y_t &= Y_0 + W_t - \alpha L_{t}, & t \in [0,T], \label{mckean-vlasov1} \\
	L_t &= \mathbb{P}(\tau \le t), & t \in [0,T], \label{mckean-vlasov_L}\\
	\tau &= \inf \{t \in [0, T]: Y_t \le 0 \} \label{mckean-vlasov2},
\end{align}
where $\alpha, T \in \mathbb{R}_{+}$, $W$ a standard Brownian motion on a probability space $(\Omega,\mathcal{F}, (\mathcal{F}_t)_{t\ge 0}, \mathbb{P})$, on which is also given an $\mathbb{R}_{+}$-valued random variable $Y_0$ independent of $W$.
The non-linearity arises from the dependence of $L_t$ in (\ref{mckean-vlasov1}) on the law of $Y$. More specifically,
if $t\rightarrow L_t$ has a derivative $p_\tau$, which is then the density of the hitting time $\tau$ of zero,
$$
dY_t = dW_t - \alpha \ p_\tau(t) \, dt, \qquad t \in (0,T),
$$
so that the drift depends on the law of the path of $Y$.

Theoretical properties of \eqref{mckean-vlasov1}--\eqref{mckean-vlasov2}  have been studied in \cite{Andreas}, who prove the existence of a differentiable solution $(L_t)_{0<t<\texp}$ up to an ``explosion time''
$\texp$. Conversely, they show that $L$ cannot be continuous for all $t$ for $\alpha$ above a threshold determined by the law of $Y_0$.
Such systemic events where discontinuities occur are also referred to as ``blow-ups'' in the literature.

The question of the constructive solution, however, remained open.
Examples of a numerical solution computed with Algorithm \ref{algo1} introduced in Section \ref{sec:main} are shown in Figure \ref{jump_err12}.
The left plot shows the formation of a discontinuity in the loss function $t\rightarrow L_t$ for increasing $\alpha$, with $Y_0 \sim \mathrm{Gamma}(1.5,0.5)$.
The density of $Y_T$ for $T$ before and after the shock is displayed in the right panel.
\begin{figure}[H]
	\begin{center}
		\subfloat[]{\includegraphics[width=0.5\textwidth]{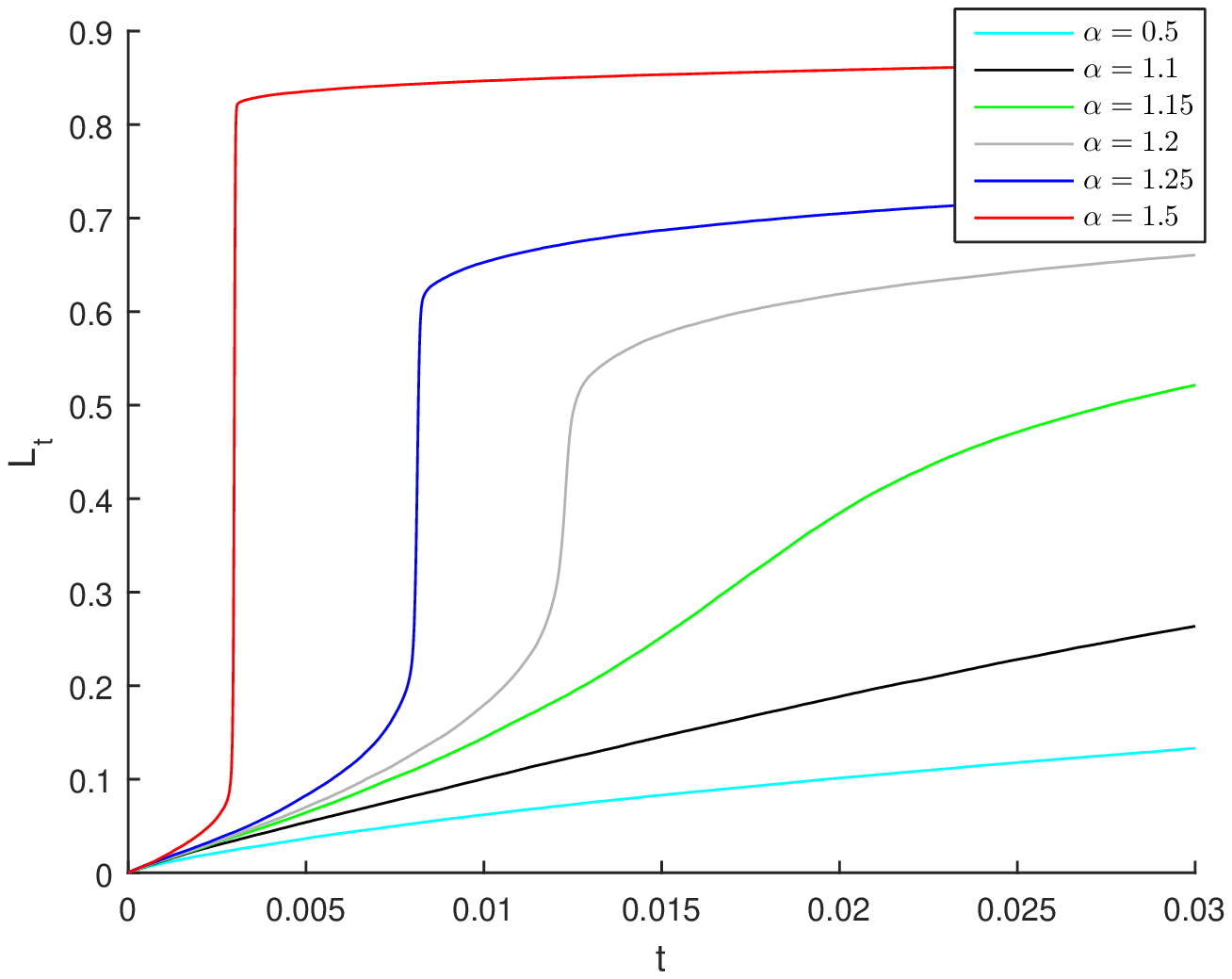}}
		\subfloat[]{\includegraphics[width=0.5\textwidth]{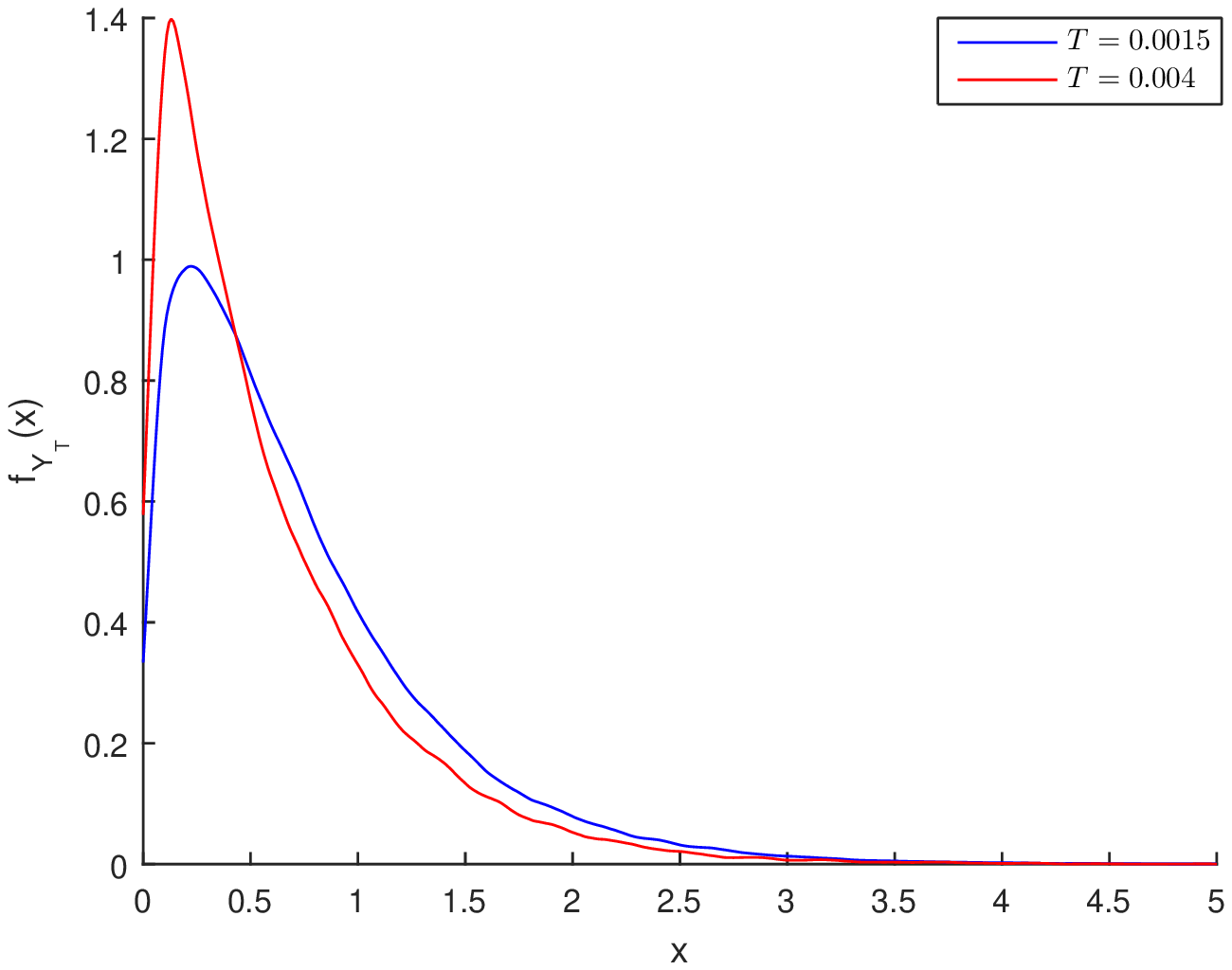}}
	\end{center}
	\vspace{-20pt}
	\caption{(a) $L_t$ for different $\alpha$ near the jump (b) Distribution of $Y_T$ for $Y_T > 0$ before and after the jump. Fitted by kernel density estimation with normal kernel for $N = 10^7$. }
 	\label{jump_err12}
\end{figure}

A similar model has been studied in \cite{Shkolnikov}, where the authors 
consider $\log(1 - L_t)$ instead of $-L_t$ in \eqref{mckean-vlasov1}. Our numerical scheme can be applied in principle to this problem, but we concentrate the analysis on \eqref{mckean-vlasov1}--\eqref{mckean-vlasov2} in this paper. 

One motivation for studying 
these equations comes from mathematical finance, in particular, systemic risk. A large interconnected banking network can be approximated by a particle system with interactions 
by which the default of one firm, modeled as the hitting of a lower default threshold of its value, causes a downward move in the firm value of others. More details can be found in \cite{Andreas} and \cite{Shkolnikov}. This model can also be viewed as the large pool limit of a structural default model for a pool of firms where interconnectivity is caused by mutual liabilities, such as in \cite{Lipton2015}.

An earlier version this problem is found in neuroscience, where a large network of electrically coupled neurons can be described by McKean--Vlasov type equations (\cite{caceres2011analysis, carrillo2013classical,  delarue2015particle, delarue2015global}). 
If a neuron's potential reaches some fixed threshold, it jumps to a higher potential level and sends a signal to other neurons.
This feedback leads to the following equations
\begin{eqnarray}
	\label{neuro_X}
		X_t &=& X_0 + \int_0^t b(X_s) \, d s + \alpha \mathbb{E}[M_t] + W_t - M_t, \\
		M_t &=& \sum_{k \ge 1} \mathbbm{1}_{[0, t]}(\tau_k), \quad \tau_k = \inf \{ t > \tau_{k-1}: X_{t-} \ge 1\},
		\label{neuro_M}
\end{eqnarray}
where $X_0 < 1$ a.s. The similarity to \eqref{mckean-vlasov1}--\eqref{mckean-vlasov2} is seen by noticing that $L_t = \mathbb{E}[\mathbbm{1}_{[0, t]}(\tau)]$
and $M_t$ in (\ref{neuro_M}) is constant between hitting times, however,
while in (\ref{neuro_X}) an upper boundary is hit and after that the value resets to zero, in our model we are interested in hitting the zero boundary (from above), and after hitting the particle's value remains zero.

While McKean--Vlasov equations are an active area of current research, 
to our knowledge, there is only a fairly small number of papers where their simulation is studied rigorously and none of these encompass the models above.


Early works include \cite{bossy1997stochastic}, which proves convergence (with order 1/2 in the timestep and inverse number of particles)
of a particle approximation to the distribution function for the measure $\nu_t$ of $X_t$ in
the classical McKean-Vlasov equation
\begin{eqnarray}
\label{mckean-vlasov}
X_t = X_0 + \int_{\mathbb{R}} \beta(X_t,u) \ \nu_t(du) \ dt +  \int_{\mathbb{R}} \alpha(X_t,u) \ \nu_t(du) \, dW_t
\end{eqnarray}
with sufficient regularity.
The proven rate in the timestep is improved to 1 in \cite{antonelli2002rate} using Malliavin calculus techniques.

More recently, multilevel simulation algorithms have been proposed and analysed:
\cite{ricketson2015multilevel} considers the special case of an SDE whose coefficients at time $t$ depend on $X_t$ and the expected value of a function of $X_t$; \cite{szpruch2017iterative} study a method based on fixed point iteration for the general case \eqref{mckean-vlasov}.
An alternative variance reduction technique by importance sampling is given in \cite{reis2018importance}.


The system \eqref{mckean-vlasov1}--\eqref{mckean-vlasov2} above does not fall into the setting of \eqref{mckean-vlasov} due to the extra path-dependence of the coefficients through the hitting time distribution.
In this paper, we therefore propose and analyse a particle scheme for \eqref{mckean-vlasov1}--\eqref{mckean-vlasov2} with an explicit timestepping scheme for the nonlinear term.
We simulate $N$ exchangeable particles at discrete time points with distance $h$, whereby at each time point $t$ we use an estimator of $L_{t-h}$ from the previous particle locations to approximate \eqref{mckean-vlasov2}. 
We prove the convergence of the numerical scheme up to some time $T$ as the time step $h$ goes to zero and number of particles goes to infinity.
The scheme can be extended up to the explosion time under certain conditions on the model parameters.
The order in $h$ for this standard estimator is 1/2. 
Next, we use Brownian bridges to better approximate the hitting probability, similar to barrier option pricing (\cite{glasserman2013mc}). In this case, the convergence rate improves to $(1 + \beta)/2$, where $\beta \in (0, 1]$ is the H\"older exponent of the density of the initial value $Y_0$ (e.g., 0.5 in the example from Figure \ref{jump_err2}).
The order can be improved to 1 by non-uniform time-stepping.

A main contribution of the paper is the first provably convergent scheme for equations of the type \eqref{mckean-vlasov1}--\eqref{mckean-vlasov2}. 
The analysis uses a direct recursion of the error and regularity results proven in \cite{Andreas}. This has the advantage that sharp convergence orders -- i.e., consistent with the numerical tests -- can be given, but also means that it seems difficult to apply the analysis directly to variations of the problem where such results are not available.
Nonetheless, the method itself is natural and applicable in principle to other settings such as those outlined above.

The rest of the paper is organized as follows. In Section \ref{sec:main} we list the running assumptions and state the main results of the paper; 
in Sections \ref{sec:conveul} and \ref{subsec:ext} we prove the uniform convergence of the discretized process; 
in Section \ref{sec:MC}, we show the convergence of Monte Carlo particle estimators with an increasing number of samples;
in Section \ref{sec:BB} we prove the convergence order for the scheme with Brownian bridge;
in Section \ref{sec:numer} we give numerical tests of the schemes;
finally, in Section \ref{sec:concl}, we conclude.

\section{Assumptions and main results}
\label{sec:main}

We begin by listing the assumptions.
The first one, H\"older continuity at 0 of the initial density, is key for the regularity of the solution. The H\"older exponent will also limit
the rate of convergence of the discrete time schemes.
\begin{assumption}
	\label{assumption_Y0}
	We assume that $Y_0$ has a density $f_{Y_0}$ supported on $\mathbb{R}_+$ such that 
	\begin{equation}
		\label{Y_0_assumption}
		f_{Y_0}(x) \le B x^{\beta}, \qquad x \ge 0
	\end{equation}
	 for some $\beta \in (0, 1]$. 
\end{assumption}
Under Assumption \ref{assumption_Y0}, we can refer to Theorem 1.8 in \cite{Andreas} for the existence of  a unique, differentiable solution $t\rightarrow L_t$ for 
\eqref{mckean-vlasov1}--\eqref{mckean-vlasov2} up to time
	\begin{equation*}
		\texp := \sup{ \{t > 0: ||L||_{H^1(0, t)} < \infty \} } \in [0, \infty],
	\end{equation*}
	and a corresponding $\hat{B}$ such that for every $t <\texp$
	\begin{equation}
		\label{L_t_prime_ineq}
	 	L_t' \le \hat{B} t^{-\frac{1-\beta}{2}} \ a.e.
	 \end{equation}
This estimate admits a singularity of the rate of losses at time 0, however, what is actually observed in numerical studies (see Figure \ref{jump_err12}, left) is that the loss rate is bounded initially but then has a sharp peak for small $\beta$ (and then especially for large $\alpha$).

Integrating \eqref{L_t_prime_ineq}, we have for future reference a bound for $L_t$,
\begin{equation}
	\label{L_t_ineq}
	L_t \le \tilde{B} t^{\frac{1 + \beta}{2}},
\end{equation}
where $\tilde{B} = 2 \hat{B}/({1 + \beta})$.

The following assumption will be used to control the propagation of the discretisation error, by bounding the density (especially at 0)
of the running minimum of $Y$ and its approximations.
\begin{assumption}
	\label{assumption_T}
	We assume that $T < \min(T^{*}, \texp)$,
	where $T^{*}$ is defined by
	\begin{equation}
		\label{theorem_condition}
		\alpha B \left[\sqrt{\frac{2 T^{*}}{\pi}} + \alpha \tilde{B} (T^{*})^{\frac{1 + \beta}{2}}  \right]^{\beta} = 1,
	\end{equation}	
	with $B$ and $\tilde{B}$ the smallest constants such that \eqref{Y_0_assumption} and \eqref{L_t_ineq} hold for given $\beta$.
\end{assumption}
In the following, we assume that Assumptions \ref{assumption_Y0} and \ref{assumption_T} hold.
Consider a uniform time mesh $0 = t_0 < t_1 < \ldots < t_n = T$, where $t_i - t_{i-1} = h$, and a discretized process,
for $1\le i\le n$,
\begin{eqnarray}
	\tilde{Y}_{t_i} &=& Y_0 + W_{t_i} - \alpha \tilde{L}_{t_i}, \label{discr_process1}\\
	\tilde{L}_{t_i} &=& \mathbb{P}(\tilde{\tau} < t_i), \\
	\tilde{\tau} &=& \min_{0 \le j \le n} \{\tilde{Y}_{t_j} \le 0\} \label{discr_process2}.
\end{eqnarray}
We extend $\tilde{L}_{t_i}$ to $[0, T]$ by setting $\tilde{L}_s = \tilde{L}_{t_{i-1}}$ for $t_{i-1} < s < t_i$.

The first theorem, proven in Section \ref{sec:conveul}, shows that $\tilde{L}_{t}$ converges uniformly to $L_{t}$. 

\begin{theorem}
	\label{explicit_theorem}
	Consider $\tilde{L}_{t_i}$ from \eqref{discr_process1}--\eqref{discr_process2} and $L_t$ from \eqref{mckean-vlasov1}--\eqref{mckean-vlasov2}. Then, for any $\delta > 0$, there exists $C > 0$ independent of $h$ such that 
	\begin{equation}
		\label{theor1_result}	
		\max_{i \le n} | \tilde{L}_{t_i} - L_{t_i} | \le C h^{\frac{1}{2} - \delta}.
	\end{equation}
\end{theorem}

We now propose a particle simulation scheme for \eqref{discr_process1}--\eqref{discr_process2} in Algorithm \ref{algo1}.
\begin{algorithm}[H]
	\caption{Discrete time Monte Carlo scheme for simulation of the loss process}
	\begin{algorithmic}[1]
		\Require{$N$ --- number of Monte Carlo paths}
		\Require{$n$ --- number of time steps: $0 < t_1 < t_2 < \ldots < t_n$}
		\State{Draw $N$ samples of $Y_0$ (from initial distribution) and  $W$ (a Brownian path)}
		\State{Define $\hat{L}_0 = 0$}
		\For{$i = 1:n$}
			\State{Estimate $\tilde{L}_{t_i}$ by $\hat{L}_{t_i}^N=\frac{1}{N}\sum_{k = 1}^N  \mathbbm{1}_{\{ \min_{j < i} \hat{Y}_{t_j}^{(k)}  \le 0 \}}$}
			\For{$k = 1:N$}
			\State{Update $\hat{Y}_{t_i}^{(k)} = Y_0^{(k)} + W_{t_i}^{(k)} - \alpha \hat{L}_{t_i}^N$}
			\EndFor
		\EndFor
	\end{algorithmic}
	\label{algo1}
\end{algorithm}

In Section \ref{sec:MC}, we prove convergence in probability of Algorithm \ref{algo1}  
as $N \to \infty$. 
\begin{theorem}
	\label{mc_theorem}
	For all $i \le n$,
	\begin{equation}
		\label{theorem_mc_result}	
		\hat{L}_{t_i} \xrightarrow[N \to \infty]{\mathbb{P}} \tilde{L}_{t_i}.
	\end{equation}
\end{theorem}

Next, we improve our scheme by using a Brownian bridge strategy to estimate the hitting probabilities. In order to do this, we consider the process
\begin{eqnarray}
	\label{discr_process1_bb} 
	\dtilde{Y}_{t} &=& Y_0 + W_{t} - \alpha \dtilde{L}_{t}, \qquad\!\!\! t \in [t_i, t_{i+1}),  \\
	\label{discr_process3_bb} 
	\dtilde{L}_{t} &=& \mathbb{P}(\dtilde{\tau} < t_i),  \qquad \qquad t \in [t_i, t_{i+1}), \\
	\label{discr_process2_bb} 
	\dtilde{\tau} &=& \inf_{0 \le s \le T} \{\dtilde{Y}_{s} \le 0\}.
\end{eqnarray}

Then, for each Brownian path $(W_{t}^{(k)})_{t\ge 0}$, we compute $\bar{Y}_{t}^{(k)} = Y_0^{(k)} + W_{t}^{(k)} - \alpha \bar{L}^N_{t}$ in $(t_i,t_{i+1})$, where $\bar{L}^N_{t}$ is an $N$-sample estimator of $\dtilde{L}_{t_i}$ given below. Hence, using Brownian bridges, we compute
\begin{eqnarray*}
	p^{(k)}_{t_i} &=& \mathbb{P} \left( \inf_{s < t_i} \bar{Y}_s^{(k)} > 0 | \bar{Y}_0^{(k)} , \ldots, \bar{Y}_{t_i}^{(k)} \right) \\
	 &=& \prod_{j = 1}^{i} \mathbb{P} \left(\inf_{s \in [t_{j-1}, t_j)} \bar{Y}^{(k)}_s > 0 \given \bar{Y}_{t_{j-1}}^{(k)} , \bar{Y}_{t_j}^{(k)} \right) \\
	  &=& \prod_{j = 1}^i \left( 1 - \exp\left(-\frac{2(\bar{Y}_{t_{j-1}}^{(k)} \vee 0)(\bar{Y}_{t_{j}-}^{(k)} \vee 0)}{h} \right) \right).
\end{eqnarray*}
Thus, a natural choice for $\bar{L}^N_{t_i}$ is
\begin{equation}
	\label{barL_est}
	\bar{L}^N_{t_i} = \frac{1}{N} \sum_{k = 1}^N \left(1 - p^{(k)}_{t_i}\right).
\end{equation}
As a result, the new algorithm with the Brownian bridge modification is the following.
\begin{algorithm}[H]
	\caption{Discrete time Monte Carlo scheme with Brownian bridge}
	\begin{algorithmic}[1]
		\Require{$N$ --- number of Monte Carlo paths}
		\Require{$n$ --- number of time steps: $0 < t_1 < t_2 < \ldots < t_n$}
		\State{Draw $N$ samples $Y_0$ (from the initial distribution) and  $W$ (a Brownian path)}
		\For{$i = 1:n$}
			\State{Estimate $\bar{L}_{t_i}$ using \eqref{barL_est}}
			\For{$k = 1:N$}
			\State{Update $\bar{Y}_{t_i}^{(k)} = Y_0^{(k)} + W_{t_i}^{(k)} - \alpha \bar{L}^N_{t_i}$}
			\EndFor
		\EndFor
	\end{algorithmic}
	\label{algo2}
\end{algorithm}
The convergence rate for (\ref{discr_process3_bb}) is given as follows and proven in Section \ref{sec:BB}.
\begin{theorem}
	\label{bb_theorem}
	Consider $\dtilde{L}_{t}$ from \eqref{discr_process1_bb}--\eqref{discr_process2_bb} and $L_t$ from \eqref{mckean-vlasov1}--\eqref{mckean-vlasov2},
	$\beta \in (0,1]$ from Assumption \ref{assumption_Y0}.
	Then, there exists $C > 0$ independent of $h$ such that 
	\begin{equation}
		\label{theor4_result}	
		\max_{i \le n} | \dtilde{L}_{t_i} - L_{t_i} | \le C h^{\frac{1 + \beta}{2}}.
	\end{equation}
\end{theorem}
We will later give a result with variable time steps which achieves rate 1 for all $\beta$.

%% file: Auxiliary.tex
\section{Convergence results}

\subsection{Convergence of the timestepping scheme}
\label{sec:conveul}

In this section we prove Theorem \ref{explicit_theorem}. Then, in Section \ref{subsec:ext}, under a modification of Assumption \ref{assumption_T},
we formulate and prove an improvement of this theorem which extends the applicable time interval.

The proof is based on induction on the error bound over the timesteps,
which requires an error estimate of the hitting probability after discretisation (Lemmas \ref{lemma_asmussen} and \ref{lemma_1_2_delta}), a sort of consistency,
plus a control of  the resulting misspecification of the barrier through a bound on the density of the running minimum (Lemma \ref{density_estimate}), a kind of stability.

As we have crude estimates on the densities, using only Assumption \ref{assumption_Y0} but no sharper
bound on, and regularity of, $f_{Y_0}$, or any regularity of the distribution of $\inf_{s\le t} (W_s-\alpha L_s)$ or its approximations, the time $T^*$ until which the
numerical scheme is shown to converge will by  no means be sharp. Indeed, in our numerical tests (see Section \ref{sec:numer}) we did not encounter any difficulties for any $t$, even $t>\texp$.

We first formulate these auxiliary results which we will use to prove Theorems \ref{explicit_theorem}, \ref{bb_theorem}, and \ref{explicit_theorem_improved}. The proofs are given in Appendix \ref{appendix_proofs}.

First, we modify Proposition 1 from \cite{asmussen1995discretization}, where an analogous result is shown for standard Brownian motion, i.e., without the presence of $L$ and $\tilde{L}$ (i.e., $\alpha=0$).
\begin{lemma}
	\label{lemma_asmussen}
	Define $Y^{*}_t = Y_0 + W_{h \lfloor \frac{t}{h} \rfloor} - \alpha L_{h \lfloor \frac{t}{h} \rfloor}$ and  $\tilde{Y}^{*}_t = Y_0 + W_{h \lfloor \frac{t}{h} \rfloor} - \alpha \tilde{L}_{h \lfloor \frac{t}{h} \rfloor}$. Then, as $h \to 0$,
	\begin{equation}
		\label{asmussen_result}
		 \frac{1}{\sqrt{h}} \sup_{s \le t} \left(Y_{s}  - Y^{*}_s \right)\rightarrow_{d}  \sqrt{2 \log{\frac{t}{h}}},
	 \end{equation}
	 and
	\begin{equation}
		\label{asmussen_result2}
		 \frac{1}{\sqrt{h}} \sup_{s \le t} \left(\tilde{Y}_{s}  - \tilde{Y}^{*}_s \right)\rightarrow_{d}  \sqrt{2 \log{\frac{t}{h}}}.
	 \end{equation}	 
\end{lemma}
The next lemma deduced the convergence rate of the hitting probabilities on the mesh.
\begin{lemma}
	\label{lemma_1_2_delta}
	Consider the processes $Y$ and $\tilde{Y}$. Then, for any $\delta > 0$ there exist $\gamma, \tilde{\gamma} > 0$, independent of $h$ and $i$, such that 
	\begin{equation}
	\label{eqn_1_2_delta}
		  0 \le \mathbb{P}\left(\min_{j < i} Y_{t_j} > 0\right) -  \mathbb{P}\left(\inf_{s < t_i} Y_s > 0\right)  \le  \gamma h^{\frac{1}{2} - \delta}
	\end{equation}
	and
%
	\label{lemma_1_2_delta_tilde}
	\begin{equation}
	\label{eqn_1_2_delta_tilde}
		  0 \le \mathbb{P}\left(\min_{j < i} \tilde{Y}_{t_j} > 0\right) -  \mathbb{P}\left(\inf_{s < t_i} \tilde{Y}_s > 0\right)  \le  \tilde{\gamma} h^{\frac{1}{2} - \delta}.
	\end{equation}
\end{lemma}

Finally, we bound the probability that the running minimum is close to the boundary.
\begin{lemma}
	Consider $Z_i = \inf_{s \le t_i} Y_{s}$, $\bar{Z}_i = \min_{j < i} Y_{t_j}$,  and $\tilde{Z}_i = \min_{j < i} \tilde{Y}_{t_j}$ for some $i \le n$. Then, $Z_i, \bar{Z}_i$, and $\tilde{Z}_i$ each have a density, denoted $\varphi_i$, $\bar{\varphi}_i$, and $\tilde{\varphi}_i$, respectively, with
	\begin{equation}
		\max(\varphi_i(z), \bar{\varphi}_i(z), \tilde{\varphi}_i(z)) \le B \left[(z \vee 0) + \sqrt{\frac{2 t_i}{\pi}} + \alpha \tilde{B} t_i^{\frac{1 + \beta}{2}}  \right]^{\beta}, 
		\label{varphi_estimate} \\
	\end{equation}
	where $\beta$, $B$ are from \eqref{Y_0_assumption} and $\tilde{B}$ is from \eqref{L_t_ineq}.	
	\label{density_estimate}	
\end{lemma}

%% file: Uniform_convergence.tex

We are now in  a position to prove Theorem \ref{explicit_theorem}. 

\begin{proof}[Proof of Theorem \ref{explicit_theorem}]
We split the error into two contributions, 
\begin{multline*}
	| \tilde{L}_{t_i} - L_{t_i} | = \left| \mathbb{P}\left(\min_{j < i} \tilde{Y}_{t_j} > 0\right) - \mathbb{P}\left(\inf_{s < t_i} Y_s > 0\right) \right| \\
	\le \left| \mathbb{P}\left(\min_{j < i} \tilde{Y}_{t_j}> 0\right) - \mathbb{P}\left(\min_{j < i}  Y_{t_j} > 0\right)  \right| + \left|  \mathbb{P}\left(\min_{j < i} Y_{t_j} > 0\right) -  \mathbb{P}\left(\inf_{s < t_i} Y_s > 0\right) \right|.
\end{multline*}
We can use Lemma \ref{lemma_1_2_delta}, (\ref{eqn_1_2_delta}), for the second term.
Now we shall proceed by induction to estimate $| \tilde{L}_{t_i} - L_{t_i} |$.
For $t_0 = 0$, we have $L_0 = \tilde{L}_0$. Assume we have shown $\tilde{L}_{t_j} = L_{t_j} - \tilde{C}_j h^{\frac{1}{2} - \delta}$ for $j < i$,
where $\tilde{C}_j \ge 0$ as $\tilde{L}_{t_j} \le L_{t_j}$. Then,
\begin{eqnarray*}
 \mathbb{P}\left(\min_{j < i} \tilde{Y}_{t_j}> 0\right) -  \mathbb{P}\left(\min_{j < i}Y_{t_j} > 0\right) && \\
 &\hspace{-4 cm}=& \hspace{-2 cm} \mathbb{P}\left(\min_{j < i} \left(Y_{t_j} + \alpha \tilde{C}_j h^{\frac{1}{2} - \delta}   \right) > 0\right) - \mathbb{P}\left(\min_{j < i} Y_{t_j} > 0\right)    \\
&\hspace{-4 cm}\le& \hspace{-2 cm} \mathbb{P}\left(\min_{j < i} Y_{t_j}   >  -\alpha \max_{j < i} {\tilde{C}_j} h^{\frac{1}{2} - \delta} \right) - \mathbb{P}\left(\min_{j < i} Y_{t_j} > 0\right) \\
&\hspace{-4 cm}=& \hspace{-2 cm} \bar{F}_i(0) - \bar{F}_i\left(  -\alpha \max_{j < i} {\tilde{C}_j} h^{\frac{1}{2} - \delta}  \right) \\
&\hspace{-4 cm}\le& \hspace{-2 cm}   \alpha \max_{j < i} { \tilde{C}_j} h^{\frac{1}{2} - \delta}  
\sup_{\theta \in [0,1]} \bar{\varphi}_i \left( -\theta \alpha \max_{j < i} { \tilde{C}_j} h^{\frac{1}{2} - \delta} \right),
\end{eqnarray*}
where $\bar{F}_i(x)$ and $\bar{\varphi}_i(x)$ are the CDF and pdf of $\min_{j < i} Y_{t_j}$. 

Then, using Lemma \ref{density_estimate}, we have
\begin{equation*}
	\bar{\varphi}_i \left( -\theta \frac{\alpha}{2} \max_{j < i} { \tilde{C}_j} h^{\frac{1}{2} - \delta} \right) \le B \left[ \sqrt{\frac{2 t_i}{\pi}} + \alpha \tilde{B} t_i^{\frac{1 + \beta}{2}}  \right]^{\beta},
\end{equation*}
as $ -\theta \frac{\alpha}{2} \max_{j < i} { \tilde{C}_j} h^{\frac{1}{2} - \delta} < 0$.
As a result, we have the following inequality for $\tilde{C}_i$,
\begin{equation*}
	\tilde{C}_i \le \alpha \max_{j < i} \tilde{C}_j B \left[ \sqrt{\frac{2 t_i}{\pi}} + \alpha \tilde{B} t_i^{\frac{1 + \beta}{2}}  \right]^{\beta} + \gamma,
\end{equation*}
%
hence $\tilde{C}_i$ is bounded independent of $i$ and $h$ by Assumption \ref{assumption_T}.
By induction we get \eqref{theor1_result}.

\end{proof}

\subsection{Extension of the result in time}
\label{subsec:ext}

The following result extends the applicability of Theorem~\ref{explicit_theorem} up to the explosion time $\texp$ under certain conditions on the parameters, 
as specified precisely in \eqref{theorem2_condition} below. 

We shall adapt Theorem 1 in \cite{Novikov}, which states: For a Lipschitz function $f$ with Lipschitz constant $K$, and $g$ such that  $\sup_{s \le t} | f(s) - g(s) | \le \varepsilon$,
\begin{equation*}
	| \mathbb{P}(\exists s \in [0, t]:  W_s < f(s)) - \mathbb{P}(\exists s \in [0, t]:  W_s < g(s)) | \le \left(2.5 K + \frac{2}{\sqrt{t}} \right) \varepsilon.
\end{equation*}
By Remark 2 in \cite{Novikov}, this result can be improved for a non-decreasing function $g$. Indeed, retracing the steps in their proof and using monotonicity, one finds easily the slightly better bound
\begin{equation*}
	| \mathbb{P}(\exists s \in [0, t]:  W_s < f(s)) - \mathbb{P}(\exists s \in [0, t]:  W_s < g(s)) | \le 2 \left( K + \frac{1}{\sqrt{t}} \right) \varepsilon.
\end{equation*}

In our case, we cannot directly apply the result with $f(s)=-Y_0 + \alpha L_s$ and $g(s)=-Y_0 + \alpha \tilde{L}_s$
as $f$ is not guaranteed to be Lipschitz at $s=0$.
But, along the lines of the proof of Theorem 1 in \cite{Novikov}, the above result can be modified as follows:

\begin{lemma}
For a non-decreasing function $f$ which is Lipschitz with constant $K$ on $[T^*, T]$, and a function $g$ such that $f(t) = g(t)$ for $t \le T^{*}$ and $\sup_{s \le T} | f(s) - g(s) | \le \varepsilon$,
\begin{equation}
	\label{novikov_theorem}
	| \mathbb{P}(\exists s \in [0, T]:  W_s < f(s)) - \mathbb{P}(\exists s \in [0, T]:  W_s < g(s)) | \le 2 \left( K + \frac{1}{\sqrt{T}} \right) \varepsilon.
\end{equation}
\end{lemma}

\begin{theorem}
	\label{explicit_theorem_improved}
Assume $T^*$ from (\ref{theorem_condition}) satisfies
	\begin{equation}
	  		 \label{theorem2_condition}	
			 2 \alpha \left(\hat{B} (T^{*})^{-\frac{1-\beta}{2}} + \frac{1}{\sqrt{T^{*}}}\right) < 1.
	\end{equation}
Then Theorem \ref{explicit_theorem} holds for any $T<t_*$ (and not only up to $T^*$ as per Assumption \ref{assumption_T}).

\end{theorem}
\begin{proof}
We first split again the error by
\begin{multline}
	\label{split_eq}
	| \tilde{L}_{t_i} - L_{t_i} | = \left| \mathbb{P}\left(\min_{j < i} \tilde{Y}_{t_j} > 0\right) - \mathbb{P}\left(\inf_{s \le t_i} Y_s > 0\right) \right| \\
	\le \left| \mathbb{P}\left(\inf_{s \le t_i} \tilde{Y}_{s} > 0 \right) - \mathbb{P}\left(\inf_{s \le t_i}  Y_{s} > 0\right)  \right| + \left|  \mathbb{P}\left(\min_{j < i} \tilde{Y}_{t_j} > 0\right) -  \mathbb{P}\left(\inf_{s \le t_i} \tilde{Y}_s > 0\right) \right|.
\end{multline}
The second term can be estimated from Lemma \ref{lemma_1_2_delta_tilde}, (\ref{eqn_1_2_delta_tilde}).
Again, we shall then proceed by induction. We have already shown that $| \tilde{L}_{t_i} - L_{t_i} | \le C_{i} h^{\frac{1}{2} - \delta}$ for $t_i \le T^{*}$ according to Theorem \ref{explicit_theorem}.

Consider $T^{*} < t_i \le T$. Assume we have shown that $|L_{t_j} - \tilde{L}_{t_j}| \le C_j h^{\frac{1}{2} - \delta}$ for $j < i$.
We want to derive $C_i$ such that 
$|L_{t_i} - \tilde{L}_{t_i}| \le  C_i h^{\frac{1}{2} - \delta}$ and where all $C_i$ are bounded independent of $h$.
First, consider an intermediate point $s \in (t_{i-1}, t_i)$, then 
\begin{eqnarray*}
	L_{s} - \tilde{L}_{s}  \le L_{t_{i-1}} + K (s - t_{i-1}) - \tilde{L}_{t_{i-1}} 
	\le   L_{t_{i-1}} - \tilde{L}_{t_{i-1}} +  K h, 
\end{eqnarray*}
with $K$ the Lipschitz constant of $L$, and thus
\begin{equation}
	\label{max_sup_ineq}
	\sup_{s \le t_i} |L_{s} - \tilde{L}_{s}|  \le \max \left(\max_{j < i} |L_{t_{j}} - \tilde{L}_{t_{j}}| +  K h, |L_{t_{i}} - \tilde{L}_{t_{i}}| \right).
\end{equation}


Now we show that $|L_{t_i} - \tilde{L}_{t_i}| \le C_{t_{i}} h^{\frac{1}{2} - \delta}$.
Consider 
\begin{equation*}
	l_t = 
	\begin{cases}
		L_t, \quad t \le T^{*}, \\
		\tilde{L}_t, \quad t > T^{*},
	\end{cases}
\end{equation*}
and $Y^l_t = Y_0 + W_t - \alpha l_t$. Then, 
\begin{multline}
 \left| \mathbb{P}\left(\inf_{s \le t_i} \tilde{Y}_{s}> 0\right) -  \mathbb{P}\left(\inf_{s \le t_i}Y_{s} > 0\right) \right|  \\
 \le  \left| \mathbb{P}\left(\inf_{s \le t_i} \tilde{Y}_{s}> 0\right) -  \mathbb{P}\left(\inf_{s \le t_i}Y^l_{s} > 0\right) \right|  +  \left| \mathbb{P}\left(\inf_{s \le t_i} Y^l_{s}> 0\right) -  \mathbb{P}\left(\inf_{s \le t_i}Y_{s} > 0\right) \right|.
 \label{eqn:twoterms}
\end{multline}
To estimate the first term, we can write
\begin{eqnarray*}
	\left| \mathbb{P}\left(\inf_{s \le t_i} \tilde{Y}_{s}> 0\right) -  \mathbb{P}\left(\inf_{s \le t_i}Y^l_{s} > 0\right) \right| 
	&=& \mathbb{E} \left[ \mathbbm{1}_{\exists t \in [0, t_i]: Y^l_{t} \le 0  \}, \forall s \in [0, t_i] \tilde{Y}_{s} > 0 } \right] \\
 &=& \mathbb{E} \left[ \mathbbm{1}_{\exists t \in [0, T^{*}]: Y^l_{t} \le 0  \}, \forall s \in [0, t_i] \tilde{Y}_{s} > 0 } \right] \\
	&\le& \mathbb{E} \left[ \mathbbm{1}_{\exists t \in [0, T^{*}]: Y^l_{t} \le 0  \}, \forall s \in [0, T^{*}] \tilde{Y}_{s} > 0 } \right] \\
	&=& \left| \mathbb{P}\left(\inf_{s \le T^{*}} \tilde{Y}_{s}> 0\right) -  \mathbb{P}\left(\inf_{s \le T^{*}}Y^l_{s} > 0\right) \right|,
\end{eqnarray*}
where we have used in the second line that
since $l_t = \tilde{L}_t$ for $t > T^{*}$, hitting after $T^{*}$ does not affect the difference.
For $t \le T^{*}$, $l_t = L_t$. Then, using Theorem \ref{explicit_theorem},
\begin{multline*}
	\left| \mathbb{P}\left(\inf_{s \le t_i} \tilde{Y}_{s}> 0\right) -  \mathbb{P}\left(\inf_{s \le t_i}Y^l_{s} > 0\right) \right| \le \left| \mathbb{P}\left(\inf_{s \le T^{*}} \tilde{Y}_{s}> 0\right) -  \mathbb{P}\left(\inf_{s \le T^{*}}Y_{s} > 0\right) \right| \le \bar{C}^{T^{*}}  h^{\frac{1}{2} - \delta}.
\end{multline*}

For the second term in (\ref{eqn:twoterms}), $L_t$ is Lipschitz on $[T^{*}, T]$ with $K = \hat{B} (T^{*})^{-\frac{1-\beta}{2}}$ and we can apply \eqref{novikov_theorem}. Thus,
\begin{eqnarray}
\label{expansion_novikov}
 \left| \mathbb{P}\left(\inf_{s \le t_i} \tilde{Y}_{s}> 0\right) -  \mathbb{P}\left(\inf_{s \le t_i}Y_{s} > 0\right) \right|  \le  2 \alpha \left(K + \frac{1}{\sqrt{t_i}}\right) \sup_{s \le t_i} | L_{s} - \tilde{L}_{s} | + \bar{C}^{T^{*}} h^{\frac{1}{2} - \delta} \\
	\label{expansion1_improved}
	 \le  2 \alpha \left(K + \frac{1}{\sqrt{t_i}}\right)\max \left(\max_{j < i} |L_{t_{j}} - \tilde{L}_{t_{j}}| +  K h, |L_{t_{i}} - \tilde{L}_{t_{i}}| \right) +  \bar{C}^{T^{*}} h^{\frac{1}{2} - \delta},
\end{eqnarray}
using \eqref{max_sup_ineq}.
Moreover, by (\ref{eqn_1_2_delta_tilde}) and \eqref{expansion1_improved}, \eqref{split_eq} can be written as
\begin{equation*}
	| L_{t_i} - \tilde{L}_{t_i}| \le  2 \alpha \left(K + \frac{1}{\sqrt{t_i}}\right)\max \left(\max_{j < i} |L_{t_{j}} - \tilde{L}_{t_{j}}| +  K h, |L_{t_{i}} - \tilde{L}_{t_{i}}| \right)+ \bar{\gamma} h^{\frac{1}{2} - \delta},
\end{equation*}
where $\bar{\gamma} = \tilde{\gamma} + \bar{C}^{T^{*}}$.

Taking into account \eqref{theorem2_condition}, we have
\begin{equation*}
	| L_{t_i} - \tilde{L}_{t_i}| \le \max \left(\left( 2 \alpha \left(K + \frac{1}{\sqrt{t_i}}\right)  \max_{j < i} C_j + \bar{\gamma} +\tilde{\varepsilon} \right) h^{\frac{1}{2} - \delta}, \frac{\bar{\gamma}}{1-  2\alpha \left(K + \frac{1}{\sqrt{t_i}}\right)} h^{\frac{1}{2} - \delta} \right),
\end{equation*}
where $\tilde{\varepsilon} =  2\alpha \left(K + \frac{1}{\sqrt{T^{*}}}\right) K  h^{\frac{1}{2} + \delta} $.

As a result, we have the following inequality for $C_i$,
\begin{equation*}
	C_i \le \max \left(\left( 2\alpha \left(K + \frac{1}{\sqrt{t_i}}\right)  \max_{j < i} C_j + \bar{\gamma} +\tilde{\varepsilon} \right), \frac{\bar{\gamma}}{1-  2\alpha \left(K + \frac{1}{\sqrt{t_i}}\right)} \right).
\end{equation*}
Hence,
%
because of \eqref{theorem2_condition},
$C_i$ is bounded independent of $h$, and by induction we get the result. 

\end{proof}

%% file: Monte_Carlo_conv.tex
\subsection{Monte Carlo simulation of discretized process}
\label{sec:MC}

In this section, we prove the convergence in probability of
\[
\hat{L}_{t_i} = \frac{1}{N}\sum_{k = 1}^N  \mathbbm{1}_{\{ \min_{j < i} \hat{Y}_{t_j}^{(k)}  \le 0 \}}
\]
in Algorithm \ref{algo1} to $\tilde{L}_{t_i}$ as $N \to \infty$. 
We note that we cannot directly apply the law of large numbers, as the summands 
are dependent through $\hat{L}_{t_j}^N, j < i$. However, we see below that the dependence diminishes (i.e., the covariance goes to zero)
as $N \to \infty$, which easily gives convergence, albeit without a Central Limit Theorem-type error estimate or a rate for the variance.

First, we formulate an auxiliary lemma.
\begin{lemma}
	\label{lemma_mc}
	\label{lemma2_mc}
	Consider $i \le n$. Assume for all $j < i$
	\begin{equation}
		\label{lemma_mc_assumption}
		\hat{L}_{t_j}^N \xrightarrow[]{\mathbb{P}} \tilde{L}_{t_j}.
	\end{equation}
	Then,
	\begin{eqnarray}
		\label{lemma_convergence_eq}
		\mathbb{E}[\hat{L}_{t_i}^N] &\xrightarrow[N \to \infty]{}& \tilde{L}_{t_i} \\
		\label{lemma2_convergence_eq}
		\mathbb{V}[\hat{L}_{t_i}^N]  &\xrightarrow[N \to \infty]{}&  0.
	\end{eqnarray}
\end{lemma}
The proof is given in Appendix \ref{lemma_mc_proof}.
Now we can deduce the convergence instantly.
\begin{proof}[Proof of Theorem \ref{mc_theorem}]
	The proof is immediate by induction. The statement is true for $i = 0$. Now take $i \ge 1$.
	By Lemma \ref{lemma_mc}, there exists $N^{*}$ such that for all $N > N^{*}$, 
	\begin{equation*}
		| \mathbb{E}[\hat{L}_{t_i}^N] - \tilde{L}_{t_i}| \le \frac{\varepsilon}{2}.
	\end{equation*}
	Thus, by Chebyshev's inequality, we have
	\begin{equation*}
		\mathbb{P}( | \hat{L}_{t_i}^N - \tilde{L}_{t_i} | > \varepsilon) \le \mathbb{P} \left( | \hat{L}_{t_i}^N - \mathbb{E}[\hat{L}_{t_i}^N] | > \frac{\varepsilon}{2}\right) \le \frac{4 \mathbb{V}[\hat{L}_{t_i}^N]}{\varepsilon^2}.
	\end{equation*}
	Using again Lemma \ref{lemma2_mc}, we have that $\mathbb{V}[\hat{L}_{t_i}^N]  \xrightarrow[N \to \infty]{}  0$. Hence,
	\begin{equation*}
		\hat{L}_{t_i}^N \xrightarrow[N \to \infty]{\mathbb{P}} \tilde{L}_{t_i},
	\end{equation*}
	for $i$ and by induction we have proved the theorem.
\end{proof}

%% file: Brownian_bridge.tex
\subsection{Brownian bridge convergence improvement}
\label{sec:BB}

In this section, we prove Theorem \ref{bb_theorem}, which ascertains the uniform convergence (in $t$) of $\dtilde{L}$ to $L$ at the improved rate.

\begin{proof}[Proof of Theorem \ref{bb_theorem}]
	We shall proceed by induction. Assume we have shown that $|\dtilde{L}_{t_j} - L_{t_j}| \le C_j h^{\frac{1 + \beta}{2}}$ for all $j < i$ with some $C_j>0$, and we want to estimate $|\dtilde{L}_{t_i} - L_{t_i}|$. First, we have
	\begin{equation*}
		\sup_{t_j \le s < t_{j+1}} |\dtilde{L}_s - L_s| 
		\le  |\dtilde{L}_{t_j} - L_{t_j}| + \hat{B} h^{\frac{1 + \beta}{2}} \le (C_j + \hat{B}) h^{\frac{1 + \beta}{2}},
	\end{equation*}
	since $ L'_{\zeta} \le \hat{B} \zeta^{-\frac{1-\beta}{2}}$.
	Now consider 
	\begin{align*}
		|\dtilde{L}_{t_i} - L_{t_i}| &= \left| \mathbb{P} \left( \inf_{s \le t_i} \dtilde{Y}_s > 0 \right) - \mathbb{P} \left( \inf_{s \le t_i} {Y}_s > 0 \right) \right|  \\
		  &= \left| \mathbb{P} \left( \inf_{s \le t_i} (Y_s + \alpha( L_s - \dtilde{L}_s)) > 0 \right) - \mathbb{P} \left( \inf_{s \le t_i} {Y}_s > 0 \right) \right| \\
		&\le \left| \mathbb{P} \left( \inf_{s \le t_i} Y_s > -\alpha \sup_{s < t_i} |\dtilde{L}_s - L_s| \right) - \mathbb{P} \left( \inf_{s \le t_i} {Y}_s > 0 \right) \right| \\
		&\le  \mathbb{P} \left( \inf_{s \le t_i} Y_s > -\alpha \max_{j < i} (C_j + \hat{B}) h^{\frac{1 + \beta}{2}} \right) - \mathbb{P} \left( \inf_{s \le t_i} {Y}_s > 0 \right) \\
		&\le \alpha \max_{j < i} (C_j + \hat{B}) 
		\sup_{\theta \in [0,1]} \varphi \left(  - \theta \alpha \max_{j < i} (C_j + \hat{B}) h^{\frac{1 + \beta}{2}}  \right)  h^{\frac{1 + \beta}{2}},
	\end{align*}
	where $\varphi_i(x)$ is the density of $\inf_{s < t_i} Y_s$. 
	
	Then, using Lemma \ref{density_estimate}, we have
	\begin{equation*}
		 C_i \le \alpha B \max_{j < i} (C_j + \hat{B})  \left[ \sqrt{\frac{2 t_i}{\pi}} + \alpha \tilde{B} t_i^{\frac{1 + \beta}{2}}  \right]^{\beta}  \le \gamma \sum_{k = 0}^{i} ( \alpha B)^k \prod_{j = 1}^k \left[\sqrt{\frac{2 t_j}{\pi}} + \alpha \tilde{B} t_j^{\frac{1 + \beta}{2}}  \right]^{\beta},
	\end{equation*}
	where $\gamma = \alpha B \hat{B}  \left[ \sqrt{\frac{2 T}{\pi}} + \alpha \tilde{B} T^{\frac{1 + \beta}{2}}  \right]^{\beta}$.
	
	Thus, $C_i$ is bounded independent of $h$ and $i$ 
%
	by \eqref{theorem_condition}.
	By induction we get \eqref{theor4_result}.
	
\end{proof}

The proof of Theorem \ref{bb_theorem} indicates that the order is limited by the behaviour of $L$ for small $t$.
The next result shows that a locally refined time mesh achieves convergence order 1 for all $\beta$.
\begin{corollary}
Consider a non-uniform time mesh $t_i = (i h)^{\frac{2}{1 + \beta}}$ for $0 \le i \le n$ with $h = T^{\frac{1 + \beta}{2}}/n$. Then, there exists $C_1 > 0$, independent of $h$, such that 
\begin{equation*}
	\label{theor4_result_improved}	
	\max_{i \le n} | \dtilde{L}_{t_i} - L_{t_i} | \le C_1 h.
\end{equation*}
\end{corollary}
\begin{proof}
	The proof follows by repeating all the steps of the proof of Theorem \ref{bb_theorem}.
\end{proof}

%% file: Numerical_results.tex
\section{Numerical experiments}
\label{sec:numer}

In this section, we demonstrate that the proven convergence orders in the timestep of $1/2$, $(1+\beta)/2$, and 1 for the different methods are indeed sharp in the case
of regular solutions;
that the empirical variance of $N$-sample estimators is $1/N$; and that the method also converges experimentally in the presence of blow-up.
To show this, we study three test cases with varying regularity of the initial data and of the loss function.

\subsection{Lipschitz initial data and no blow-up}

In our first experiment,
we choose $Y_0$ such that $\frac{1}{Y_0} \sim \exp(\lambda)$, which guarantees that the density decays exponentially near zero.
We take $\lambda = 1$ in our experiments, and pick the parameters in (\ref{mckean-vlasov1}) to be $\alpha=0.8$, $T=2$.
The solution is found to be continuous.

We perform numerical simulations using Algorithms \ref{algo1} and \ref{algo2} with $N = 2 \times 10^5$ particles and different time meshes varying from $50$ to $3200$ points.
To estimate the error, we consider the difference $|\tilde{L}_T^{2n}-\tilde{L}_T^n|$ between the solutions with $n$ and $2n$ timesteps, respectively,
computed with the same paths.
The results, presented in Figure \ref{loss_fig1}, agree with the theory: for Algorithm \ref{algo1}, we get the convergence rate $\frac{1}{2}$, and for Algorithm \ref{algo2}, the rate is $1$, because the initial distribution is regular enough around $0$, i.e.\ $\beta =1$ in \eqref{assumption_Y0}.
\begin{figure}[h]
	\begin{center}
		\subfloat[]{\includegraphics[width=0.5\textwidth]{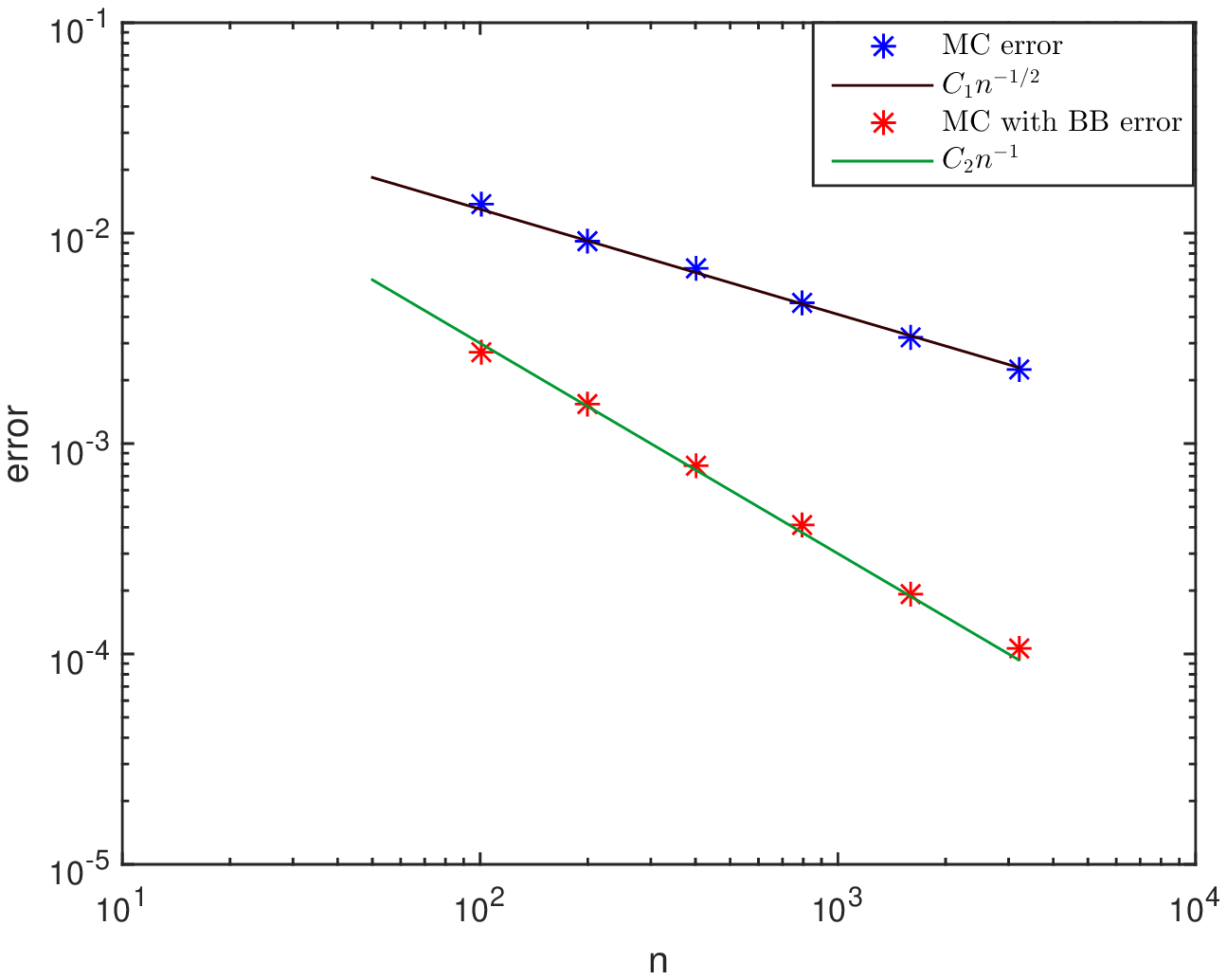}}
		\subfloat[]{\includegraphics[width=0.5\textwidth]{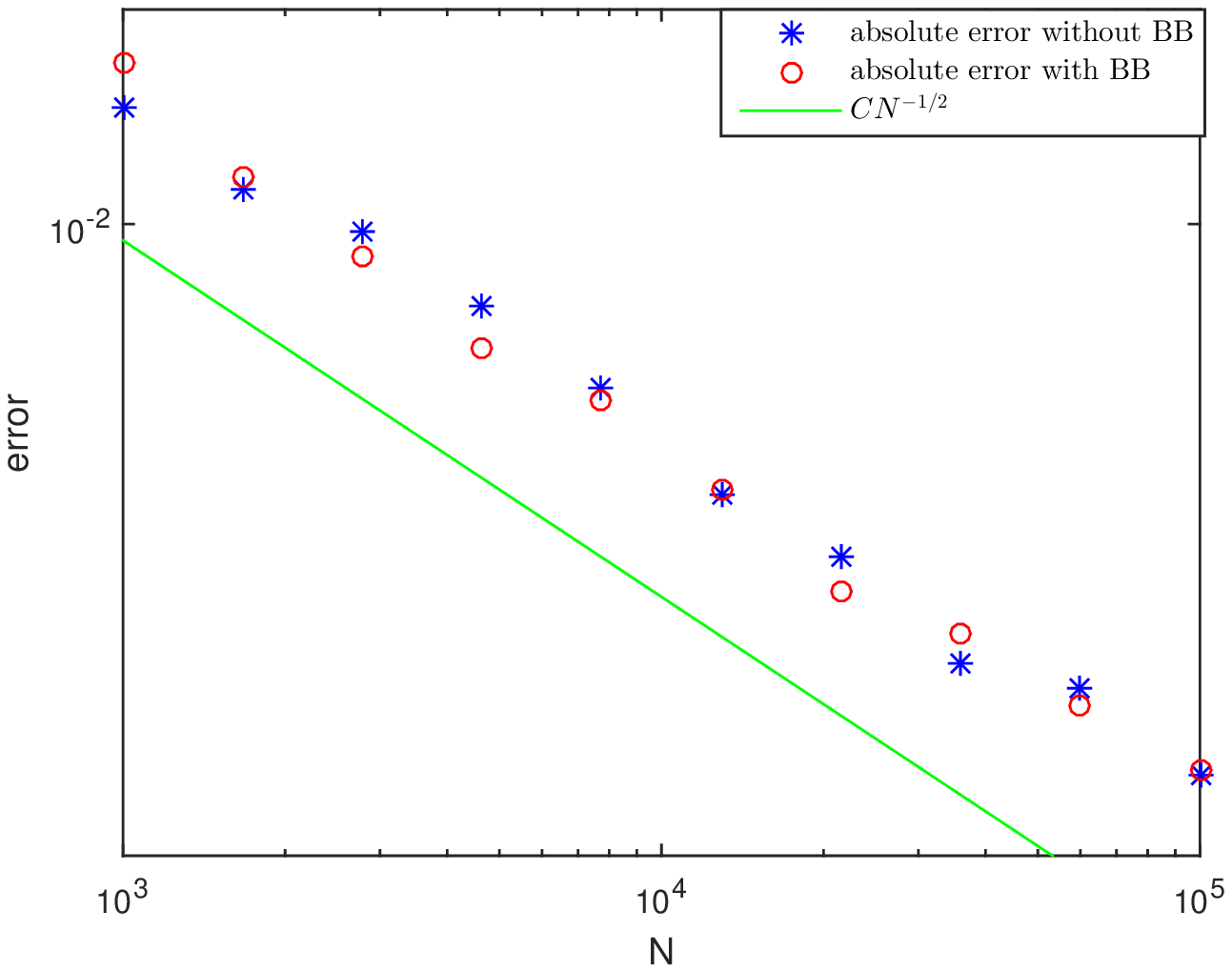}}\\
	\end{center}
	\vspace{-10pt}
	\caption{Error of the loss process at $t=T$ for $\frac{1}{Y_0} \sim \exp(1)$: 
	(a) for increasing number $n$ of timesteps; 
	(b) for increasing number $N$ of samples, 
	both for Algorithms \ref{algo1} and \ref{algo2}.}
 	\label{loss_fig1}\label{err_fig}
\end{figure}
We also investigate the convergence rate of $\hat{L}_{t_i}$ to $\tilde{L}_{t_i}$ and $\bar{L}_{t_i}$ to $\dtilde{L}_{t_i}$ empirically. 
In order to compute the benchmark solution, we used $N = 5 \times 10^7$ particles. From the results we conclude that both Algorithm \ref{algo1} and \ref{algo2} have the convergence rate $\frac{1}{2}$ in $N$.

To illustrate the dependence on the parameter $\alpha$, we include plots for $L_t$ and $L'_t$ for different values of $\alpha$. We evaluate $L'_t$ numerically using a central finite difference approximation. In order to reduce the Monte Carlo noise, we increase $N$ to $5 \times 10^{7}$ and reduce $n$ to $200$.
\begin{figure}[h]
	\begin{center}
		\subfloat[]{\includegraphics[width=0.5\textwidth]{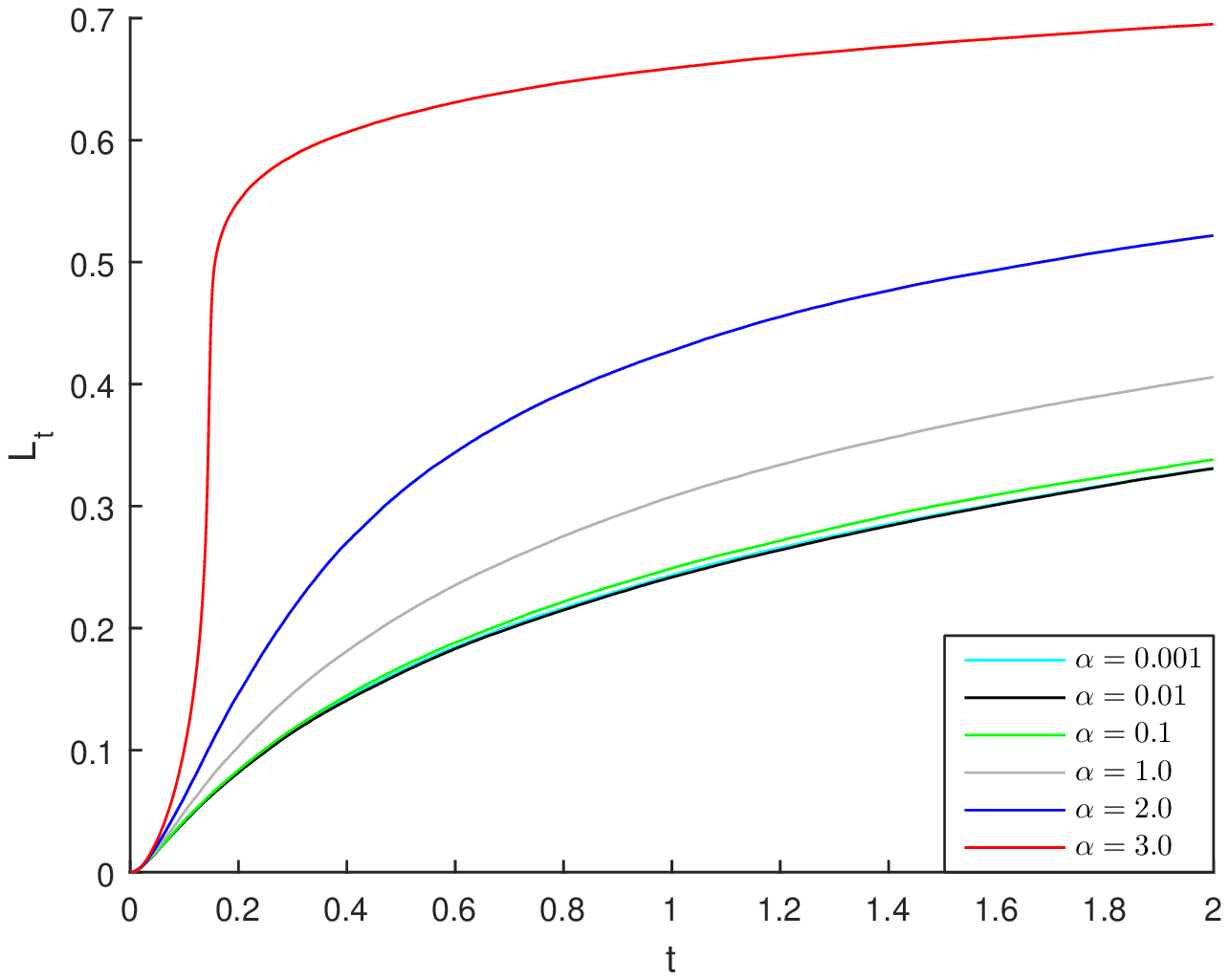}}
		\subfloat[]{\includegraphics[width=0.5\textwidth]{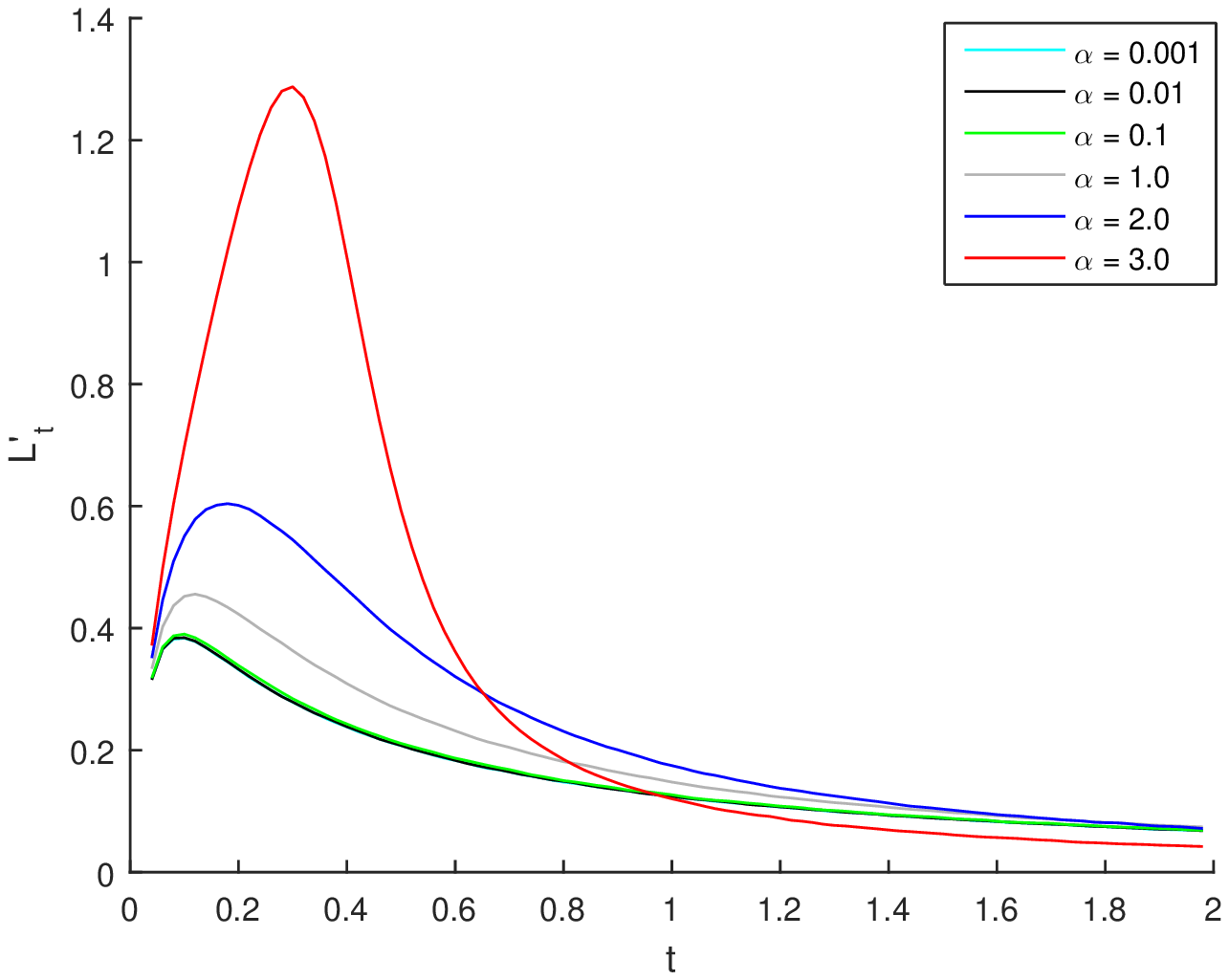}}\\
	\end{center}
	\vspace{-10pt}
	\caption{$L_t$ and $L'_t$ for different values of $\alpha$.}
 	\label{err_fig_alpha}
\end{figure}

\subsection{H\"older $1/2$ initial data and no blow-up}

In another example we again consider $\alpha=0.8$ and $T=2$, but choose $Y_0 \sim \gammadistr(1 + \beta, 1/2)$, such that we have that $f_{Y_0}(x) \le C x^{\beta}$ for $x> 0$ and some $C>0$. We choose $\beta = \frac{1}{2}$. The solution is again found to be continuous.

We perform numerical simulations using Algorithm \ref{algo1}, and Algorithm \ref{algo2} on uniform and non-uniform meshes varying from $50$ to $3200$ points 
and with $N = 2\times 10^5$ particles.
The results are presented in Figure \ref{loss_fig2}. As predicted by the theory, for Algorithm \ref{algo1} we get the convergence rate $\frac{1}{2}$, for Algorithm \ref{algo2} on uniform meshes rate $\frac{1+\beta}{2} = \frac{3}{4}$, and for Algorithm \ref{algo2} on non-uniform meshes rate $1$.
\begin{figure}[h]
	\begin{center}
		\subfloat[]{\includegraphics[width=0.525\textwidth]{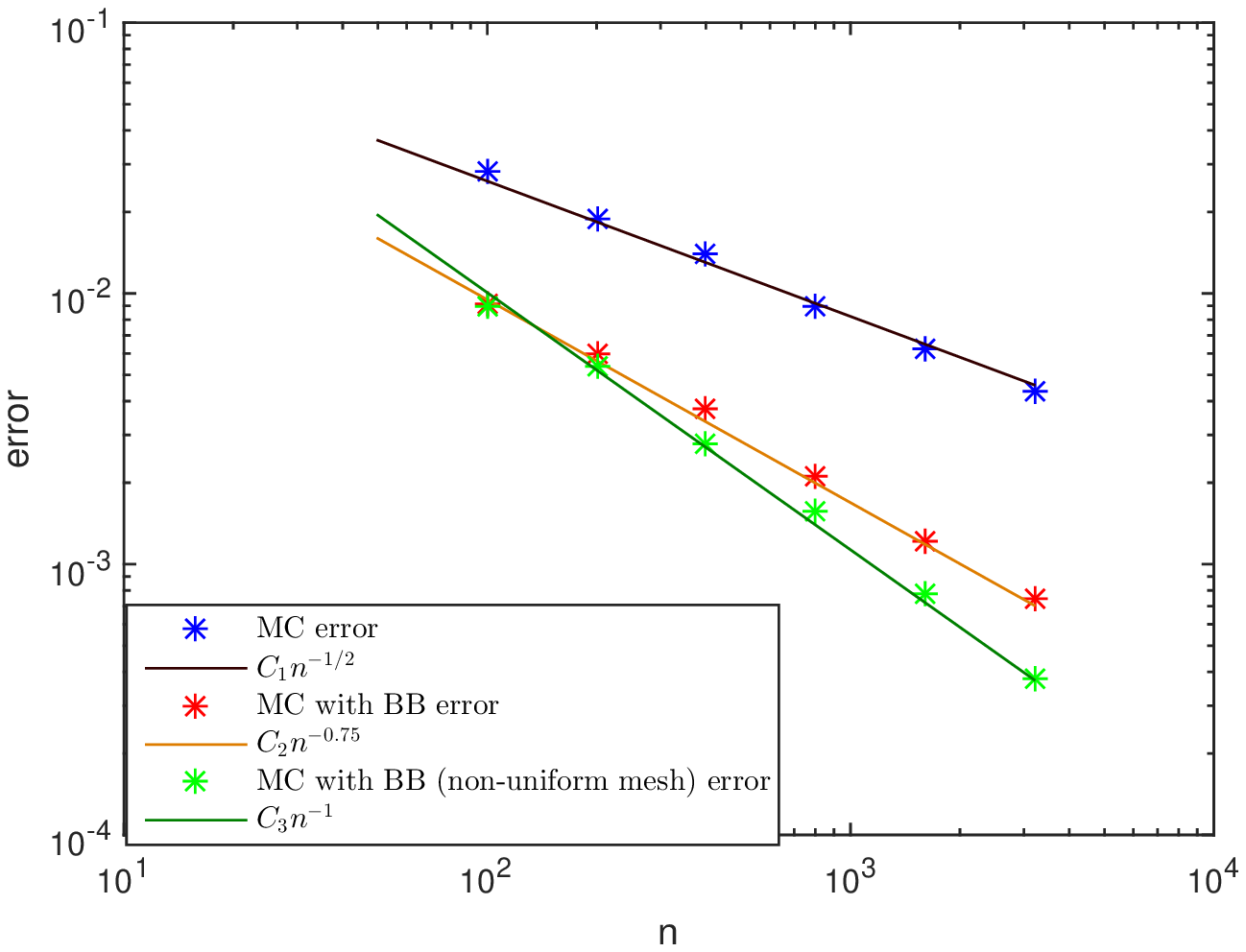}}
		\subfloat[]{\includegraphics[width=0.5\textwidth]{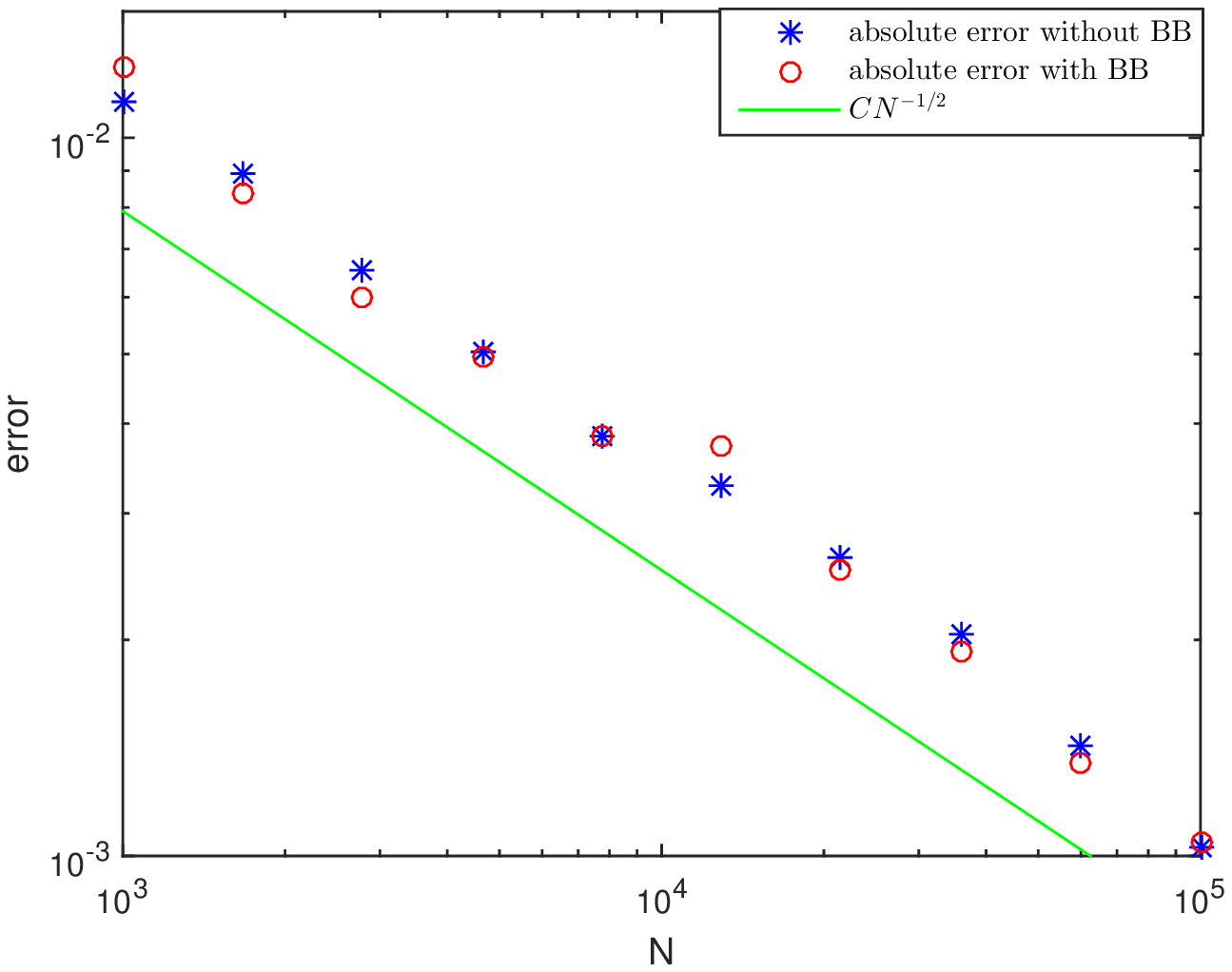}} \\
	\end{center}
	\vspace{-10pt}
	\caption{Error of the loss process at $t=T$ for $Y_0 \sim \gammadistr(3/2, 1/2)$: 
	(a) for increasing number $n$ of timesteps; 
	(b) for increasing number $N$ of samples, 
	both for Algorithms \ref{algo1} and \ref{algo2}.}
 	\label{loss_fig2}\label{err_fig2}
\end{figure}	
As in the previous example, we also investigate the convergence rate in $N$ empirically. 
These results also confirm $\frac{1}{2}$ convergence rate in $N$ for both Algorithms \ref{algo1} and \ref{algo2}.
In Figure \ref{err_fig2_alpha} we present the dependence of $L_t$ and $L'_t$ on the parameter $\alpha$. As in the previous example, we use $N = 5 \times 10^{7}$ and $n = 200$.
\begin{figure}[h]
	\begin{center}
		\subfloat[]{\includegraphics[width=0.5\textwidth]{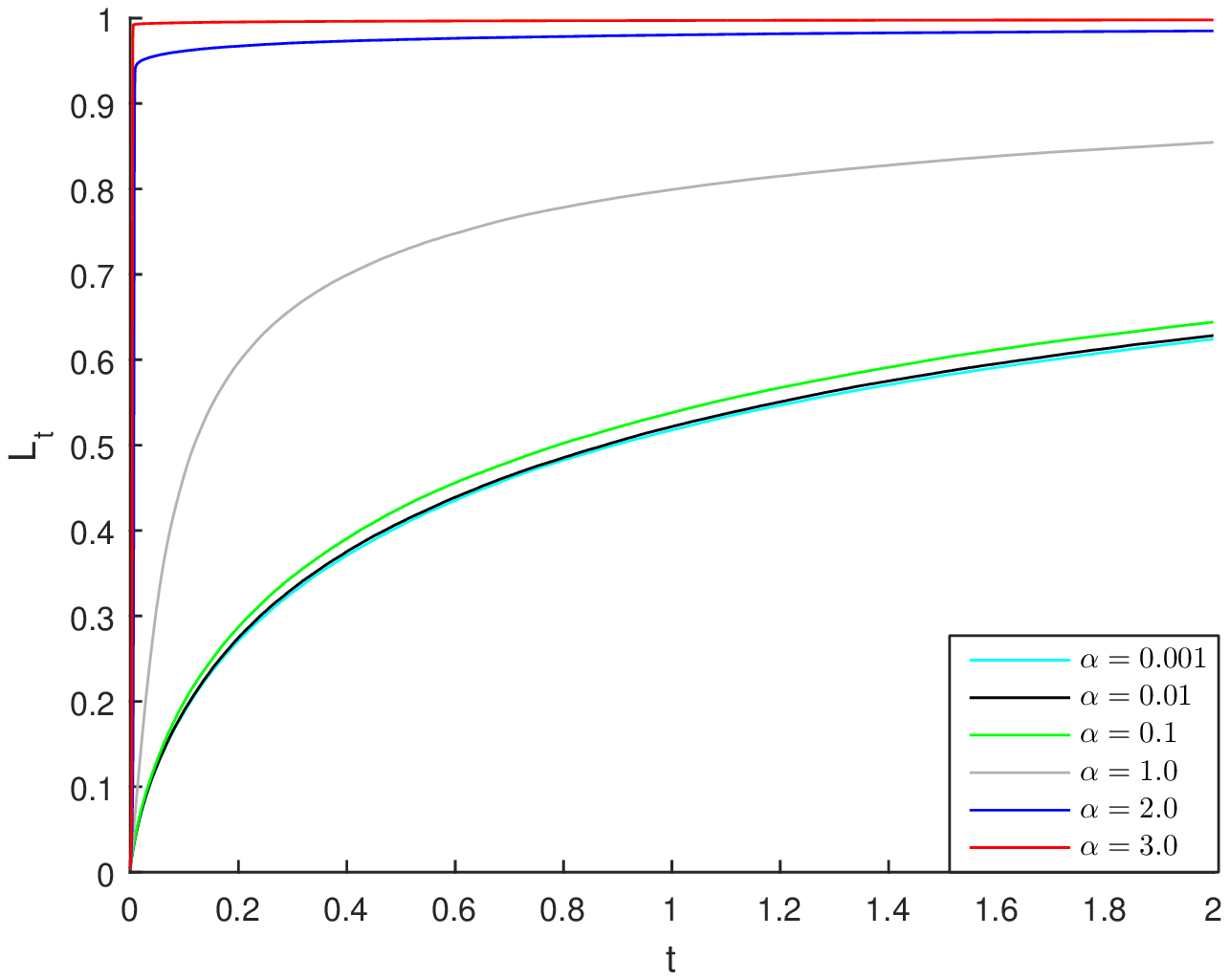}}
		\subfloat[]{\includegraphics[width=0.5\textwidth]{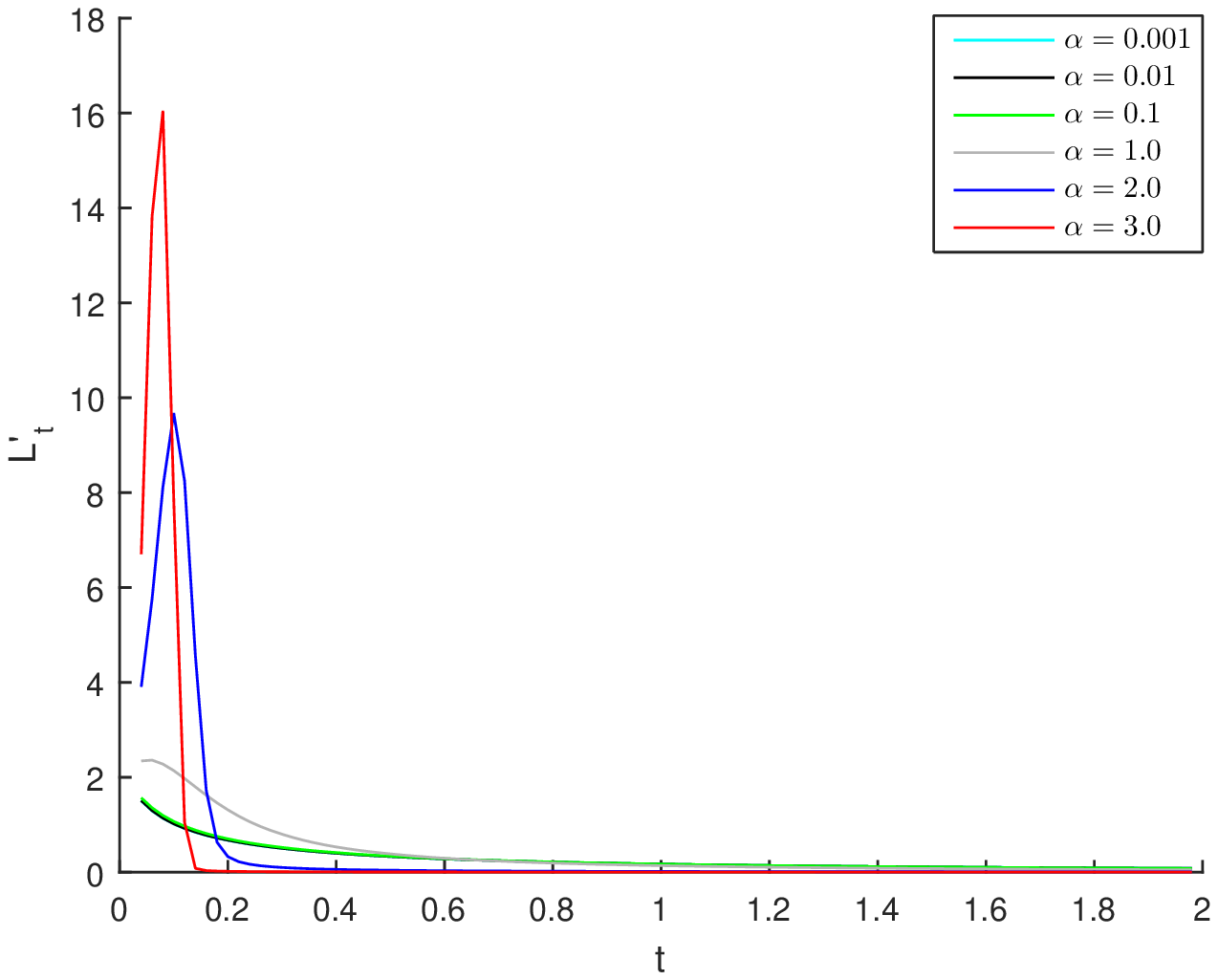}}\\
	\end{center}
	\vspace{-10pt}
	\caption{$L_t$ and $L'_t$ for different values of $\alpha$.}
 	\label{err_fig2_alpha}
\end{figure}

\subsection{H\"older $1/2$ initial data and blow-up}
\label{subsec:blowup}

In our third example, we illustrate possible jumps arising in the solution for sufficiently large values of $\alpha$. 
We consider $\alpha = 1.5$, $T = 0.008$ and choose $Y_0$ as in the previous example, $Y_0 \sim \gammadistr(1 + \beta, 1/2)$.
Note that the blow-up happens already for very small $t$ due to the interplay of the mass close to 0 for $Y_0$ and the relatively large $\alpha$.

With the lack of convergence theory for discontinuous $(L_t)_{t\ge 0}$, we apply Algorithms \ref{algo1} and \ref{algo2}, and empirically estimate the error.
In Figure \ref{jump_plot} (a) we show $\tilde{L}_t$ computed using Algorithm \ref{algo1} for different $n$;
in Figure  \ref{jump_plot} (b) the numerical error as a function of $t$ for specific $n$;
and in Figure \ref{jump_err} (a) and (b) we estimate the convergence rate for different $t$.

Figure \ref{jump_plot} (a) shows that a fairly fine resolution is needed to capture the discontinuity and its timing, but that all meshes predict the size of the jump well.
This is further illustrated in Figure \ref{jump_plot} (b), which shows the build-up of the error before the jump, the lack of uniform convergence due to the displacement
of the jump on different meshes, and the relatively constant error after the jump.

In Figure \ref{jump_err} (a) we estimate the convergence order at $T = 0.002$, i.e.\ before the jump. 
By regression, we get $0.53\ (0.47, 0.59)$  for Algorithm \ref{algo1} and $0.80\ (0.72, 0.88)$ for Algorithm \ref{algo2}, where the 95\% confidence interval is in brackets, which agrees with the theory for continuous $L_t$.
In Figure \ref{jump_err} (b), for  the error  at $T = 0.008$, i.e.\ after the jump, we get $0.93 \ (0.81, 1.05)$  for Algorithm \ref{algo1} and $1.03 \ (0.92, 1.14)$ for Algorithm \ref{algo2}. 
The faster convergence may result from the almost constancy of the losses after the jump.
\begin{figure}[h]
	\begin{center}
		\subfloat[]{\includegraphics[width=0.55\textwidth]{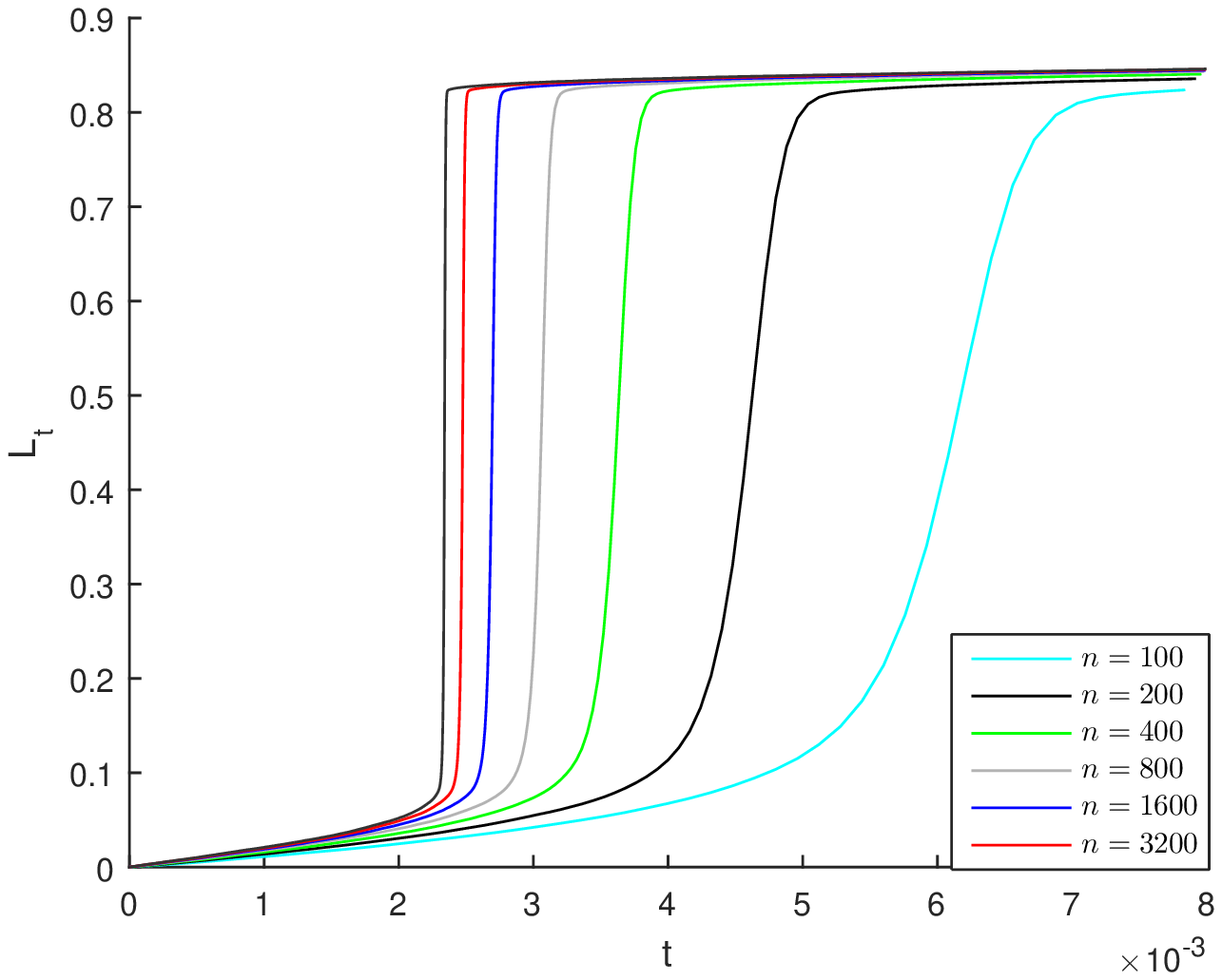}}
		\subfloat[]{\includegraphics[width=0.55\textwidth]{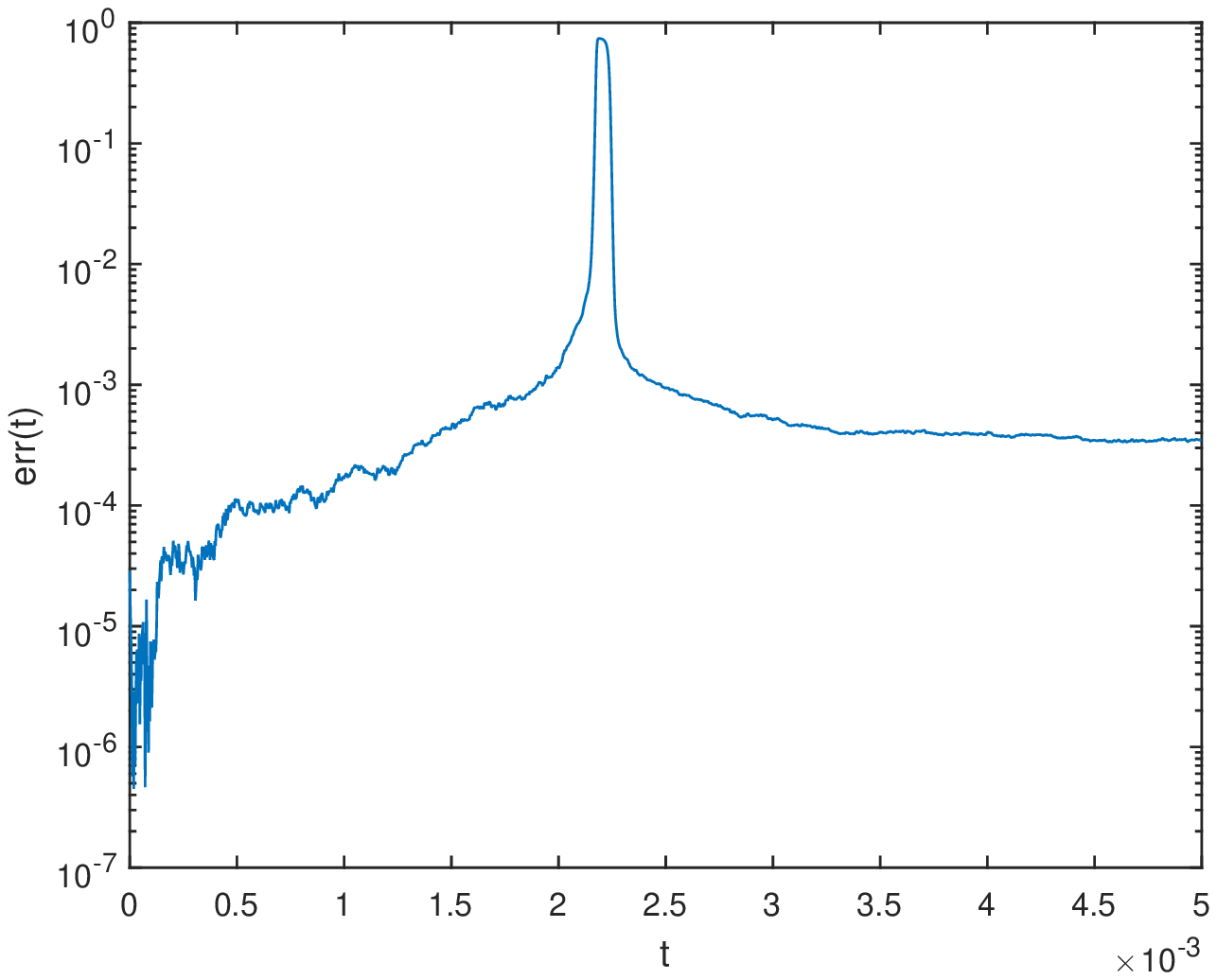}}
	\end{center}
	\vspace{-10pt}
	\caption{(a) Loss process computed using Algorithm \ref{algo1} for different $n$, (b) error as a function of $t$.}
 	\label{jump_plot}
\end{figure}

\begin{figure}[h]
	\begin{center}
		\subfloat[]{\includegraphics[width=0.55\textwidth]{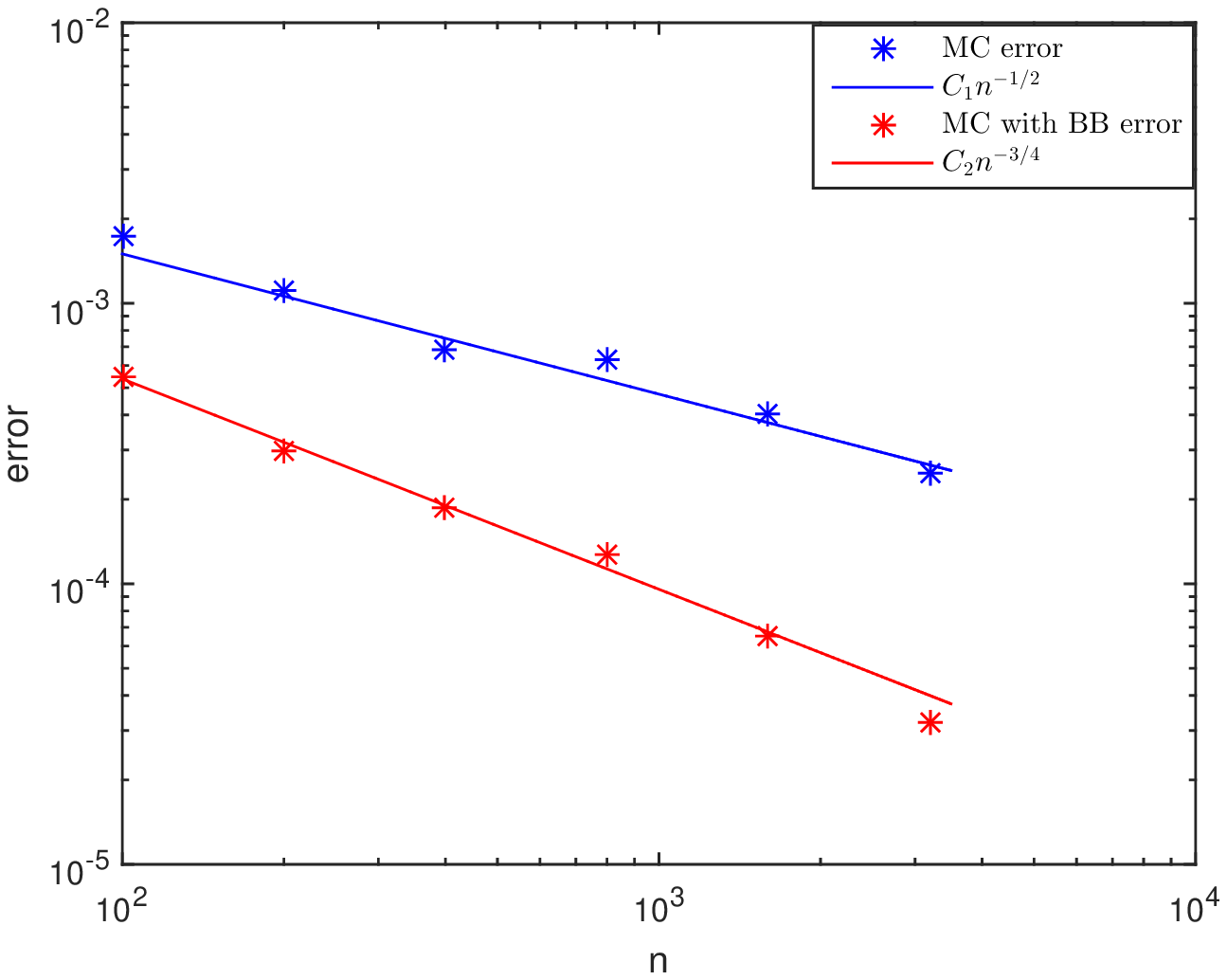}}
		\subfloat[]{\includegraphics[width=0.55\textwidth]{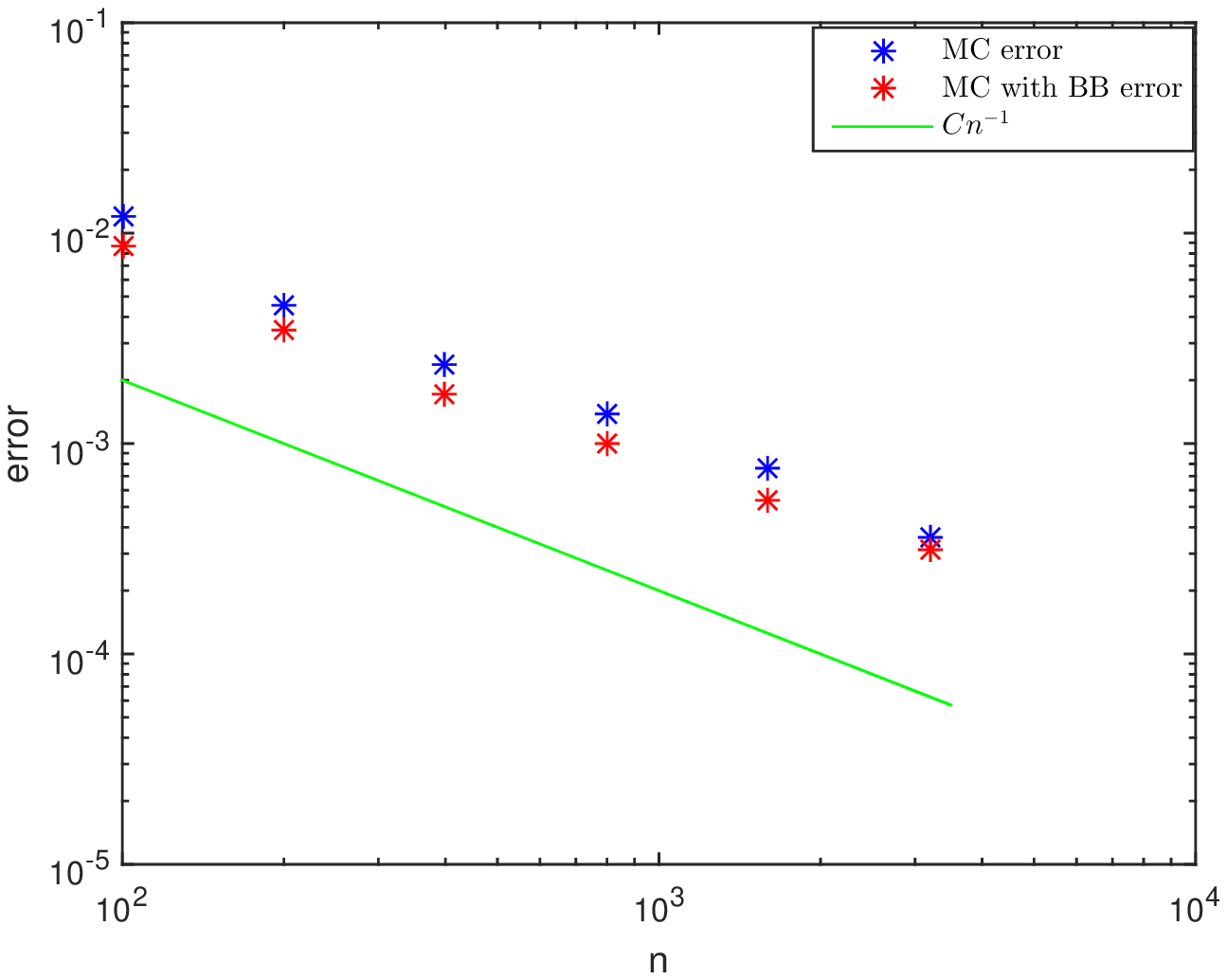}}\\
	\end{center}
	\vspace{-10pt}
	\caption{Convergence rate: at (a) $T = 0.001$, (b) $T = 0.008$.}
 	\label{jump_err}
\end{figure}

	
	To get an overall picture of the accuracy,
	we suggest the following metrics to measure the ``closeness'' of the computed solutions to $L_t$:
	\begin{enumerate}
		\item $d_1(L_t, \tilde{L}_t) = \int_0^T |L_t - \tilde{L}_t | \, d t$. This is practically computed by numerical integration.\\
		\item $d_2(L_t, \tilde{L}_t) = | t_{*} - \tilde{t}_{*} |$, where $t_{*}$ and $\tilde{t}_{*}$ are the jump times for $L_t$ and $\tilde{L}_t$, respectively.  They are approximated by the points with the steepest gradient $j_{*} = \argmax_{j} (\tilde{L}_{t_j}-\tilde{L}_{t_{j-1}}), t_* = j_{*} h$.\\
		\item $d_3(L_t, \tilde{L}_t) = \sup_{t \in [0, T]} |L^{-1}_t - \tilde{L}^{-1}_t|$, where $L^{-1}_t$ and $\tilde{L}^{-1}_t$ are the corresponding inverse functions. The inverse functions are found by the chebfun toolbox (\cite{Trefethen}), which automatically splits functions into intervals of continuity.\\
	\end{enumerate}
	In the absence of the exact $L_t$, we use again the difference between quantities 
	computed using $2n$ and $n$ points, for the same Monte Carlo paths, as a proxy.

We present the results in Figure \ref{jump_err2} (a) for Algorithm \ref{algo1} and in Figure \ref{jump_err2} (b) for Algorithm \ref{algo2}. For metric 1 we have convergence rate
$0.68 \ (0.63, 0.73)$ and $0.76 \ (0.71, 0.81)$, for metric 2 we have $0.70 \ (0.65, 0.80)$ and $0.76 \ (0.72, 0.80)$, and for metric 3 we have $0.60 \ (0.52, 0.68)$ and $0.71 \ (0.63, 0.79)$  for Algorithms \ref{algo1} and \ref{algo2}, respectively, where the 95\% confidence interval are in brackets. 
We observe that the convergence rate for Algorithm \ref{algo1} is somewhat better than $0.5$, while for Algorithm \ref{algo2} it is around $0.75$ for all three metrics.

\begin{figure}[H]
	\begin{center}
		\subfloat[]{\includegraphics[width=0.55\textwidth]{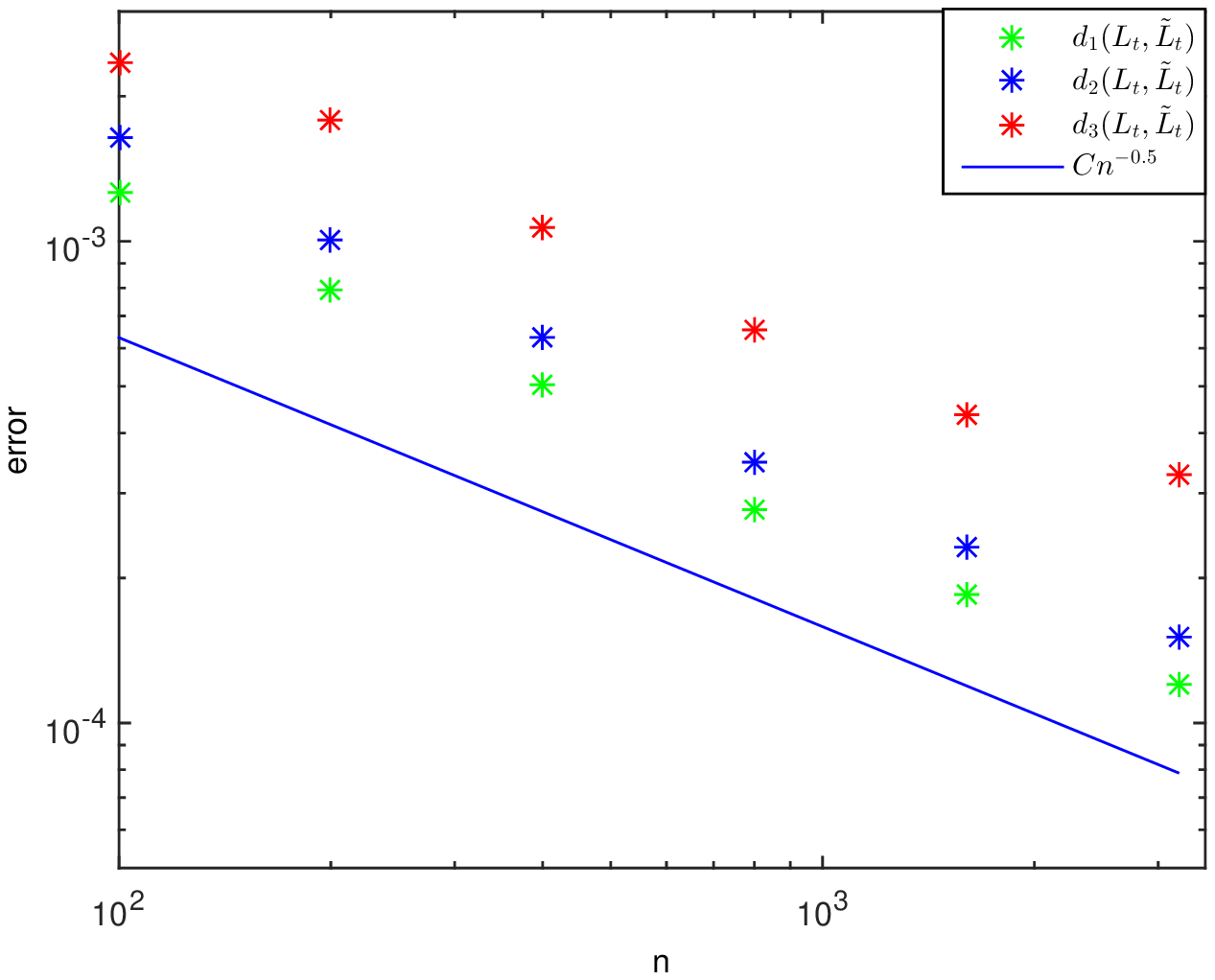}}
		\subfloat[]{\includegraphics[width=0.55\textwidth]{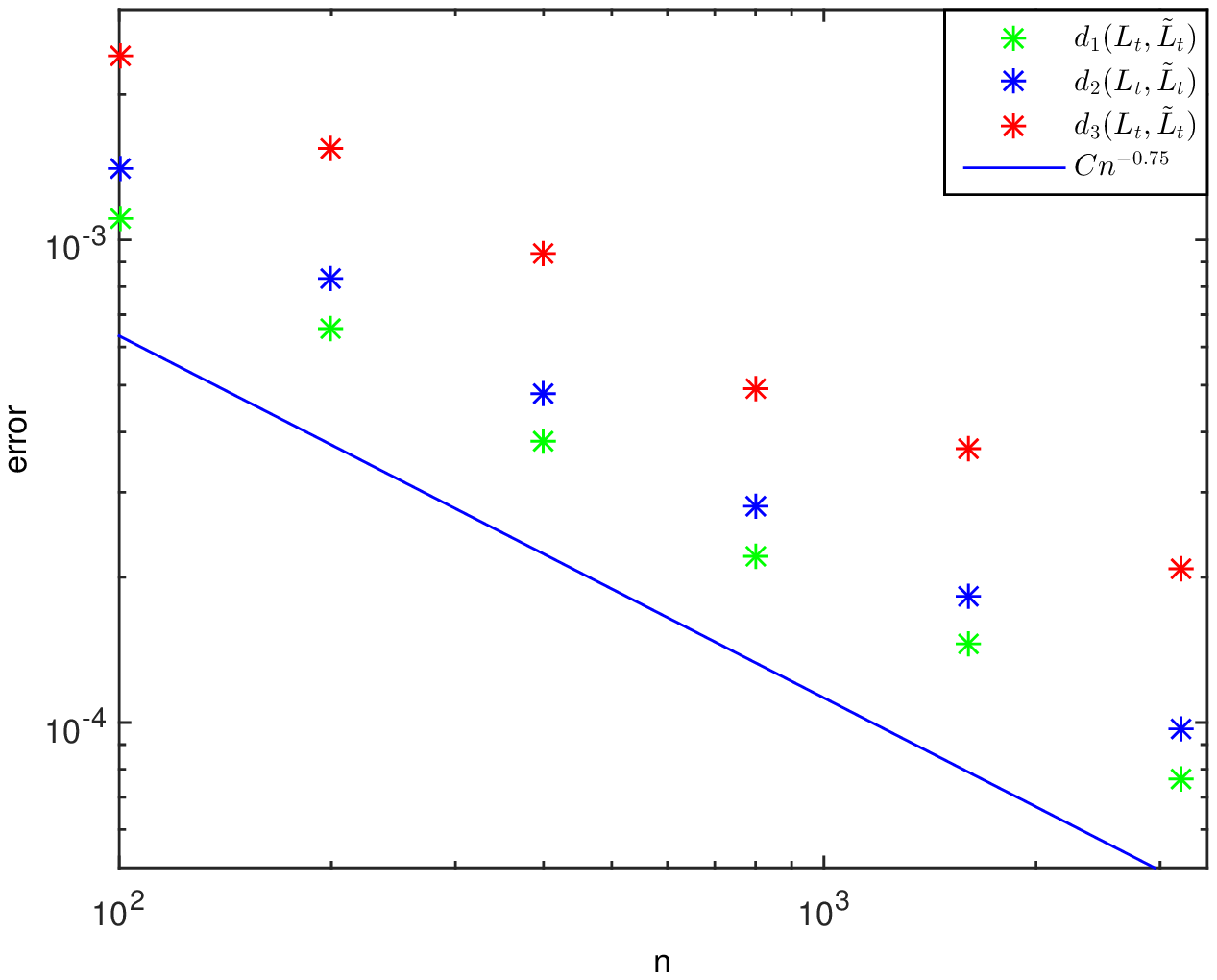}}\\
	\end{center}
	\vspace{-10pt}
	\caption{Convergence rate for $d_1(L_t, \tilde{L}_t)$, $d_2(L_t, \tilde{L}_t)$, $d_3(L_t, \tilde{L}_t)$ for  (a) Algorithm \ref{algo1}, (b) Algorithm \ref{algo2}.}
 	\label{jump_err2}
\end{figure}

%% file: Conclusion.tex
\section{Conclusions}
\label{sec:concl}

We have developed particle methods with explicit timestepping for the simulation of \eqref{mckean-vlasov1}-\eqref{mckean-vlasov2}.
Convergence with a rate up to 1 in the timestep is shown under a condition on the model parameters and time horizon, when the loss function is differentiable.
Experimentally, the method also converges in the blow-up regime, and the variance of estimators is inversely proportional to the number of samples.


This opens up several theoretical and practical questions. The efficiency of the method could be significantly improved by a simple application of multilevel simulation (\cite{szpruch2017iterative, ricketson2015multilevel}).
If the variance can be shown to behave on the finer levels as suggested by the numerical tests, combined with the proven result on the time stepping bias,
the computational complexity for root-mean-square error $\epsilon$ would be brought down to $\epsilon^{-2}$, from $\epsilon^{-3}$ as observed presently ($\epsilon^{-2}$ from the number of samples and $\epsilon^{-1}$ from the number of timesteps).

For the particular system \eqref{mckean-vlasov1}--\eqref{mckean-vlasov2}, it is conceivable to apply a particle method without time stepping of the form
\begin{eqnarray*}
Y_t^{(i)} &=& Y_0^{(i)} + W_t^{(i)} - \alpha \frac{1}{N} \sum_{j=1}^N \mathbbm{1}_{\{\tau_{(j)} \le t\}},
\end{eqnarray*}
where $Y_0^{(i)}$ and $W^{(i)}$ are $N$ i.i.d.\ copies of $Y_0$ and $W$, and $\tau_{(j)}$, $1\le j\le N$, is the order statistic of the hitting times of zero. Then, for $\tau_{(i)} < t <  \tau_{(i+1)}$, conditional on the information up to $\tau_{(i)}$, the law of $Y_t^{(j)}-Y_{\tau_{(i)}}^{(j)}$ for those $N-i$ particles $j$ which have not yet hit, is identical to that of independent standard Brownian motions with constant drift $-\alpha i/N$. So if we simulate those independent hitting times of $-Y^{(j)}_{\tau_{(i)}}$ for each and then take the smallest one to be $\tau_{(i+1)}$, this inductive construction satisfies the correct law. The complexity is then $O(N^2)$, and combined with the observed variance of $O(1/N)$, this gives a complexity of $\epsilon^{-4}$ for error $\epsilon$.
Another disadvantage is that this method with exact sampling is restricted to the constant parameter case, while the time stepping scheme can easily be extended to variable coefficients.

Theoretically, one would like guaranteed convergence also in the blow-up regime. This requires the choice of an appropriate metric -- the Skorokhod distance may be suitable, considering the analysis in \cite{delarue2015particle} and the tests in Section \ref{subsec:blowup}.
It also requires verification that the limit in the case of blow-up is a so-called ``physical solution'' as defined in \cite{Andreas, Shkolnikov,delarue2015particle}, which on the one hand results from a specific sequential realisation of the losses in the case of blow-up (\cite{Shkolnikov,delarue2015particle}), and on the other hand is the right-continuous solution with the smallest jump size (\cite{Andreas}).

Lastly, it would be interesting to investigate the extension to the models in \cite{Shkolnikov} and \cite{delarue2015particle} in more detail.

%% file: Appendix.tex
\newpage
\begin{appendices}

\section{Proof of Lemmas in Section \ref{sec:conveul}}
\label{appendix_proofs}

\subsection{Proof of Lemma \ref{lemma_asmussen}}
\begin{proof}[Proof of Lemma \ref{lemma_asmussen}]
Writing
	\begin{multline*}
		\sup_{s \le h \lfloor \frac{t}{h} \rfloor} \left(Y_{s}  - Y^{*}_s \right) 
		= \max_{1 \le k \le \lfloor \frac{t}{h} \rfloor} \sup_{0 \le s \le h} \left( W_{h (k - 1)+s} - W_{h (k-1) } - \alpha (L_{h(k-1) + s} - L_{h(k-1) } ) \right) \\
		= \max_{1 \le k \le \lfloor \frac{t}{h} \rfloor} \sup_{0 \le s \le h} \left( W_{h (k - 1) + s} - W_{h (k-1)} - \alpha s L'_{h(k-1) + \theta s} \right)
	\end{multline*}
	for some $\theta \in [0,1]$, the last expression can be estimated from both sides,
	\begin{multline*}
		 \max_{1 \le k \le \lfloor \frac{t}{h} \rfloor}  \left\{\sup_{0 \le s \le h} \left( W_{h (k - 1) + s} - W_{h (k-1)}  \right)- \alpha  \sup_{0 \le s \le h}  h L'_{h(k-1) + \theta s} \right\}  \le \\
		 \max_{1 \le k \le \lfloor \frac{t}{h} \rfloor} \sup_{0 \le s \le h} \left( W_{h (k - 1)} - W_{h (k-1) + s} - \alpha s L'_{h(k-1) + \theta s} \right)  \\
		 \le \max_{1 \le k \le \lfloor \frac{t}{h} \rfloor}  \left\{\sup_{0 \le s \le h} \left( W_{h (k - 1)} - W_{h (k-1) + s}  \right)- \alpha  \inf_{0 \le s \le h}  s L'_{h(k-1) + \theta s} \right\}.
	\end{multline*}
	
	Since $0 \le L'_s \le \hat{B} s^{-\frac{1-\beta}{2}}$,
	\begin{multline*}
		\sup_{s \le h \lfloor \frac{t}{h} \rfloor} \left(Y_{s}  - Y^{*}_s \right) 
		\ge \max_{1 \le k \le \lfloor \frac{t}{h} \rfloor}  \left\{\sup_{0 \le s \le h} \left( W_{h (k - 1) + s} - W_{h (k-1) }  \right)- \alpha \hat{B} h (h(k-1) )^{-\frac{1-\beta}{2}} \right\} \\
		=_d \sqrt{h} \max_{1 \le k \le \lfloor \frac{t}{h} \rfloor}\left\{  \inf_{0 \le s \le 1} W_s^{(k)} -\alpha \hat{B} \sqrt{h} (h(k-1) )^{-\frac{1-\beta}{2}} \right\},
	\end{multline*}
	and
	\begin{multline*}
		\sup_{s \le h \lfloor \frac{t}{h} \rfloor} \left(Y_{s}  - Y^{*}_s \right) 
		\le \max_{1 \le k \le \lfloor \frac{t}{h} \rfloor}  \left\{\sup_{0 \le s \le h} \left( W_{h (k - 1) + s} - W_{h (k-1) }  \right)\right\} \\
		=_d \sqrt{h} \max_{1 \le k \le \lfloor \frac{t}{h} \rfloor}\left\{  \sup_{0 \le s \le 1} W_s^{(k)} \right\},
	\end{multline*}
	where $W^{(1)}, W^{(2)}, \ldots$ are i.i.d. copies of $W$. 
	
	Taking into account that $ \sqrt{h} (h(k-1) )^{-\frac{1-\beta}{2}} \to 0$, as $h \to 0$, and applying the same arguments as in Proposition 1 in \cite{asmussen1995discretization}, we get \eqref{asmussen_result}.
	
	Similar,
	\begin{multline*}
		\sup_{s \le h \lfloor \frac{t}{h} \rfloor} \left(\tilde{Y}_{s}  - \tilde{Y}^{*}_s \right) 
		= \max_{1 \le k \le \lfloor \frac{t}{h} \rfloor} \sup_{0 \le s \le h} \left( W_{h (k - 1) + s} - W_{h (k-1) }   \right) \\
		= \max_{1 \le k \le \lfloor \frac{t}{h} \rfloor} \sup_{0 \le s \le h} \left( W_{h (k - 1) + s} - W_{h (k-1) }  \right)
		=_d \sqrt{h} \max_{1 \le k \le \lfloor \frac{t}{h} \rfloor}\left\{  \sup_{0 \le s \le 1} W_s^{(k)}  \right\},
	\end{multline*}
	where $W^{(1)}, W^{(2)}, \ldots$ are i.i.d. copies of $W$. 
	
	Again, applying the same arguments as in Proposition 1 in \cite{asmussen1995discretization}, we get \eqref{asmussen_result2}.	
\end{proof}

\subsection{Proof of Lemma \ref{lemma_1_2_delta}}
\begin{proof}[Proof of Lemma \ref{lemma_1_2_delta}]
We first prove \eqref{eqn_1_2_delta}.
It is obvious that
\begin{equation*}
	\mathbb{P}\left(\min_{j < i} Y_{t_j} > 0\right) \ge \mathbb{P}\left(\inf_{s < t_i} Y_s > 0\right).
\end{equation*}
Also,
\begin{equation*}
	\mathbb{P}\left(\min_{j < i} Y_{t_j} > 0\right) = \mathbb{P}\left(\inf_{s < t_i}Y_s > 0\right) + \mathbb{P}\left(\left\{\inf_{s < t_i} Y_s \le 0 \right\} \cap \left\{ \min_{j < i}Y_{t_j} > 0 \right\} \right).
\end{equation*}
Thus, 
\begin{multline*}
	\mathbb{P}\left(\min_{j < i} Y_{t_j} > 0\right) - \mathbb{P}\left(\inf_{s < t_i} Y_s > 0\right) = \mathbb{P}\left(\left\{\inf_{s < t_i}Y_s  \le 0 \right\} \cap \left\{ \min_{j < i} Y_{t_j} > 0 \right\} \right)  \\
	\le \mathbb{P}\left( \min_{j < i} Y_{t_j} -  \inf_{s < t_i} Y_s \ge \varepsilon \right) +  \mathbb{P}\left( \min_{j < i} Y_{t_j} \in (0, \varepsilon),  \min_{j < i} Y_{t_j} -  \inf_{s < t_i} Y_s <  \varepsilon \right),
\end{multline*}
for any $\varepsilon > 0$.

Consider the first term. Using Markov's inequality
\begin{equation*}
	\mathbb{P}\left( \min_{j < i} Y_{t_j} -  \inf_{s < t_i} Y_s \ge \varepsilon \right) \le \frac{\mathbb{E}\left[ \left(  \min_{j < i} Y_{t_j} -  \inf_{s < t_i} Y_s \right) \right]}{\varepsilon^p},
\end{equation*}
for any $p \ge 1$. 

Using Lemma \ref{lemma_asmussen},
\begin{equation*}
	 \frac{1}{\sqrt{h}} \left(\min_{j < i} Y_{t_j}  - \inf_{s < t_i} Y_s \right)\rightarrow_{d}  \sqrt{2 \log{\frac{t_i}{h}}}
\end{equation*}
thus
\begin{equation*}
	\frac{\mathbb{E}\left[\left(   \min_{j < i} Y_{t_j} -  \inf_{s < t_i} Y_s \right)^p\right]}{\varepsilon^p \cdot h^{p/2}} \longrightarrow_{d} \frac{(2 \log{\frac{t_i}{h}})^{p/2}}{\varepsilon^p}.
\end{equation*}
For the second term
\begin{multline*}
\mathbb{P}\left( \min_{j < i} Y_{t_j} \in (0, \varepsilon),  \min_{j < i} Y_{t_j} -  \inf_{s < t_i}Y_s <  \varepsilon \right) \le  \mathbb{P}\left( \min_{j < i} {Y}_{t_j} \in (0, \varepsilon)\right)= \\
       \mathbb{P}\left( \min_{j < i} Y_{t_j} > 0 \right) - \mathbb{P}\left( \min_{j < i} Y_{t_j} > \varepsilon \right) = \bar{F}_i(\varepsilon) - \bar{F}_i(0) \le \varepsilon \sup_{\theta \in [0, 1]} \bar{\varphi}_i \left( \theta \varepsilon \right),
\end{multline*}
where $\bar{F}_i(x)$ and $\bar{\varphi}_i(x)$ are the CDF and PDF of the process $ \min_{j < i} Y_{t_j}$. We also note that $\bar{\varphi}_i \left( \theta \varepsilon \right)$ is bounded according to Lemma \ref{density_estimate}.

Combining both terms, we have
\begin{equation*}
\mathbb{P}\left(\min_{j < i}Y_{t_j} > 0\right) - \mathbb{P}\left(\inf_{s < t_i} Y_s > 0\right) \le A \frac{h^{p/2}}{\varepsilon^p} + B \varepsilon.
\end{equation*}
Minimising the right-hand side over $\varepsilon$, we find $\varepsilon=h^a$ with $a = \frac{p}{2(p+1)}$.
As a result, choosing $p \to \infty$, we get that for any $\delta > 0$, there exists $\gamma > 0$, such that
\begin{equation*}
\mathbb{P}\left(\min_{j < i} Y_{t_j} > 0\right) - \mathbb{P}\left(\inf_{s < t_i} Y_s > 0\right) \le \gamma h^{\frac{1}{2} - \delta}.
\end{equation*}


For \eqref{eqn_1_2_delta_tilde}, we use identical steps and the corresponding estimates in Lemma \ref{lemma_asmussen} and Lemma \ref{density_estimate} for $\tilde{Y}$
instead of $Y$.

\end{proof}

\subsection{Proof of Lemma \ref{density_estimate}}
\begin{proof}[Proof of Lemma \ref{density_estimate}]
We can rewrite  $Z_i = Y_0 + V_i$, where $ V_i = \inf_{s \le t_i} (W_{s} - \alpha L_{s})$. As $Y_0$ and $V_i$ are independent, using convolution, we have
\begin{equation*}
 \varphi_i(z) = \int_{-\infty}^{+\infty} f_{Y_0}(z-v) \mathbb{P}(V_i \in dv).
\end{equation*}

Using the fact that $V_i \le 0, Y_0 \ge 0,$ and  \eqref{Y_0_assumption}, we have
\begin{multline}
 \varphi_i(z) = \int_{-\infty}^{z \wedge 0} f_{Y_0}(z-v) \mathbb{P}(V_i \in dv) \le B  \int_{-\infty}^{z \wedge 0} (z-v)^{\beta} \mathbb{P}(V_i \in dv)\\
  =  B F_{V_i}(z \wedge 0) \int_{-\infty}^{z \wedge 0} (z-v)^{\beta} \frac{ \mathbb{P}(V_i \in dv)}{F_{V_i}(z \wedge 0)},
  	\label{lemma_phi_expectation_estimate}
\end{multline}
where $F_{V_i}$ is the CDF of $V_i$. It is obvious that $F_{V_i}(z) > 0$ for all $z  \le 0$. Indeed, $F_{V_i}(z) \ge  \mathbb{P}(V_i \le z) \ge \mathbb{P}(\inf_{s \le t_i} W_{s} \le z) > 0$.

Since $\beta \in (0, 1]$, for all $z$ the function $(z-\cdot)^{\beta}: (-\infty, z \wedge 0) \to (0, \infty)$ is concave and, using Jensen's inequality for the proper probability measure $\frac{\mathbb{P}(V_i \in dv)}{F_{V_i}(z \wedge 0)}$ on $(-\infty, z \wedge 0)$,
\begin{multline}
	\label{varphi_intermediate_est}
	\varphi_i(z) \le B F_{V_i}(z \wedge 0)^{1 - \beta} \left( \int_{-\infty}^{z \wedge 0} (z-v)\mathbb{P}(V_i \in dv) \right)^{\beta} \\
	= B F_{V_i}(z \wedge 0)^{1 - \beta}  \left( \mathbb{E} \left[ (z - V_i) \mathbbm{1}_{\{V_i \le z\}} \right] \right)^{\beta},
\end{multline}
where $(V_i\le z) \Leftrightarrow (V_i\le z \wedge 0)$ as $V_i \le 0$.
Since $t\rightarrow L_t$ is non-decreasing, inserting $V_i$, 
\begin{equation}
	\begin{aligned}
		\label{varphi_ineq}
		\varphi_i(z)  
		&\le B F_{V_i}(z \wedge 0)^{1 - \beta} (\mathbb{E}[(z - \inf_{s \le t_i} \{W_{s} \} + \alpha L_{t_i}) \mathbbm{1}_{\{\inf_{s \le t_i} \{W_{s} \} - \alpha L_{t_i} \le z\}}])^{\beta} \\
		&\le  B F_{V_i}(z \wedge 0)^{1 - \beta} \left(\mathbb{E}[(z - \inf_{s \le t_i} \{W_{s} \} + \alpha \tilde{B} t_i^{\frac{1+\beta}{2}}) \mathbbm{1}_{\big\{\inf_{s \le t_i} \{W_{s} \}  \le z +  \alpha \tilde{B} t_i^{\frac{1+\beta}{2}}\big\}}]\right)^{\beta}.
	\end{aligned}
\end{equation}
By simple integration of the density of the running minimum of Brownian motion, 
\begin{equation}
	\mathbb{E}[(\xi - \inf_{s \le t_i} W_s) \mathbbm{1}_{\{\inf_{s \le t_i} W_s \le \xi\}} ] = \left(2 \Phi \left( \frac{\xi}{\sqrt{t_i}}\right) \wedge 1 \right) \xi + \sqrt{\frac{2 t_i}{\pi}} e^{-\frac{\xi^2}{2t_i}},
\end{equation}
and,  taking $\xi = z + \alpha \tilde{B} t_i^{\frac{1+\beta}{2}}$, we get from \eqref{varphi_ineq}
for $z > 0$,
\begin{equation*}
	\varphi_i(z) \le B \left[z +  \alpha \tilde{B} t_i^{\frac{1 + \beta}{2}} +  \sqrt{\frac{2 t_i}{\pi}} \right]^{\beta},
\end{equation*}
and, similarly, for $z \le 0$,
\begin{equation*}
	\varphi_i(z) \le B \left[ \alpha \tilde{B} t_i^{\frac{1 + \beta}{2}} +  \sqrt{\frac{2 t_i}{\pi}} \right]^{\beta}.
\end{equation*}
Combing the last two equations, we finally get \eqref{varphi_estimate} for $\varphi_i$.

Using similar arguments for $\bar{Z}_i$ and $\tilde{Z}_i$, 
using $\inf_{s\le t_i} W_s \le \min_{j<i} W_{t_j}$ and $L_{t_i} \ge \tilde{L}_{t_i}$, respectively,
we get the analogous estimates for $\bar{\varphi}_i$ and $\tilde{\varphi}_i$, and hence \eqref{varphi_estimate}.
\end{proof}

\section{Proof of Lemma \ref{lemma_mc}, Section \ref{sec:MC}} 
\label{lemma_mc_proof}

\begin{proof}[Proof of Lemma \ref{lemma_mc}]
	We can write
	\begin{equation*}
		\mathbb{E}[\mathbbm{1}_{\{\min_{j < i} \hat{Y}_{t_j}^{(k)} > 0\}}] = \mathbb{E}[\mathbbm{1}_{\{\min_{j < i} \{Y_0^{(k)} + W_{t_j}^{(k)} - \alpha \tilde{L}_{t_j} + \alpha(\tilde{L}_{t_j} - \hat{L}_{t_j}^N)\} > 0 \}}],
	\end{equation*}
	and estimate
	\begin{equation}
		\label{lemma_estimate}
		 \mathbb{E}[\mathbbm{1}_{\{\min_{j < i} \tilde{Y}_{t_j}^{(k)}\ > -\alpha\min_{j < i}(\tilde{L}_{t_j} - \hat{L}_{t_j}^N)\}}]  \le \mathbb{E}[\mathbbm{1}_{\{\min_{j < i} \hat{Y}_{t_j}^{(k)} > 0\}}] \le \mathbb{E}[\mathbbm{1}_{\{\min_{j < i} \tilde{Y}_{t_j}^{(k)} > -\alpha\max_{j < i} (\tilde{L}_{t_j} - \hat{L}_{t_j}^N)\}}].
	\end{equation}
	We evaluate left- and right-hand side of the last equation separately. We start with the right-hand side,
	and define for brevity $\tilde{A}^{(k)}_i = \{ \min_{j < i} \tilde{Y}_{t_j}^{(k)} >  -\alpha \max_{j < i} (\tilde{L}_{t_j} - \hat{L}_{t_j}^N) \}$. Consider some $\varepsilon > 0$.Then,
	\begin{multline*}
		\mathbb{E}[\mathbbm{1}_{\tilde{A}^{(k)}_j }] = \mathbb{E}[\mathbbm{1}_{\tilde{A}^{(k)}_j  \cap ([\max_{j < i} (\tilde{L}_{t_j} - \hat{L}_{t_j}^N) \} \le \varepsilon] \cup [\max_{j < i}  (\tilde{L}_{t_j} - \hat{L}_{t_j}^N) \} > \varepsilon] )} ] = \\
		 \mathbb{E}[\mathbbm{1}_{\tilde{A}^{(k)}_j  \cap (\max_{j < i} (\tilde{L}_{t_j} - \hat{L}_{t_j}^N)  \le \varepsilon) } ] + \mathbb{E}[\mathbbm{1}_{\tilde{A}^{(k)}_j  \cap (\max_{j < i} (\tilde{L}_{t_j} - \hat{L}_{t_j}^N)  > \varepsilon)  }] \le \\ 
		 \mathbb{P} \left( \tilde{A}^{(k)}_j  \cap (\max_{j < i} (\tilde{L}_{t_j} - \hat{L}_{t_j}^N)  \le \varepsilon)\right) + \mathbb{P} \left( \max_{j < i} (\tilde{L}_{t_j} - \hat{L}_{t_j}^N)  > \varepsilon \right) \le \\
		 \mathbb{P} \left(  \min_{j < i} \tilde{Y}_{t_j}^{(k)} >  -\alpha  \varepsilon \right) + \mathbb{P} \left( \max_{j < i} (\tilde{L}_{t_j} - \hat{L}_{t_j}^N)  > \varepsilon \right) \le \\
		 1-\tilde{L}_{t_i} + \alpha  \varepsilon  \sup_{\theta \in [0, 1]}  \tilde{\varphi}_i \left( -\theta \alpha \varepsilon \right)+ \mathbb{P} \left( \max_{j < i} (\tilde{L}_{t_j} - \hat{L}_{t_j}^N) \} > \varepsilon \right),
	\end{multline*}
	where  $\tilde{\varphi}_i(x)$  is the pdf of $\min_{j < i} \tilde{Y}_{t_j}^{(k)}$, which is bounded according to Lemma \ref{density_estimate}.
	
	Using \eqref{lemma_mc_assumption} and properties of convergence in probability, we have
	\begin{equation*}
		\max_{j < i}(\tilde{L}_{t_j} - \hat{L}_{t_j}^N) \xrightarrow[]{\mathbb{P}} 0.
	\end{equation*}
	Thus, the last term goes to $0$ with $N \to \infty$. Considering $\varepsilon \to 0$, we get 
	\begin{equation}
		\label{hat_L_bound_right}
		\lim_{N \to \infty} \mathbb{E}[\mathbbm{1}_{\tilde{A}^{(k)}_j }] \le1 - \tilde{L}_{t_i}.
	\end{equation}
	Similar, denote by  $\bar{A}^{(k)}_i = \{ \min_{j < i} \tilde{Y}_{t_j}^{(k)} >  -\alpha \min_{j < i} (\tilde{L}_{t_j} - \hat{L}_{t_j}^N) \}$, then
	\begin{multline*}
		\mathbb{E}[\mathbbm{1}_{\bar{A}^{(k)}_j }] = \mathbb{E}[\mathbbm{1}_{\tilde{A}^{(k)}_j  \cap ([\min_{j < i} (\tilde{L}_{t_j} - \hat{L}_{t_j}^N) \} \ge -\varepsilon] \cup [\min_{j < i}  (\tilde{L}_{t_j} - \hat{L}_{t_j}^N) \} < -\varepsilon] )} ] = \\
		 \mathbb{E}[\mathbbm{1}_{\bar{A}^{(k)}_j  \cap (\min_{j < i} (\tilde{L}_{t_j} - \hat{L}_{t_j}^N)  \ge -\varepsilon) } ] + \mathbb{E}[\mathbbm{1}_{\bar{A}^{(k)}_j  \cap (\min_{j < i} (\tilde{L}_{t_j} - \hat{L}_{t_j}^N)  < -\varepsilon)  }] \ge \\ 
		 \mathbb{P} \left( \bar{A}^{(k)}_j  \cap (\min_{j < i} (\tilde{L}_{t_j} - \hat{L}_{t_j}^N)  \ge -\varepsilon)\right)  \ge \\
		 \mathbb{P} \left(  \min_{j < i} \tilde{Y}_{t_j}^{(k)} >  \alpha  \varepsilon \right) \ge 1-\tilde{L}_{t_i} -  \alpha  \varepsilon   \inf_{\theta \in [0, 1]} \tilde{\varphi}_i \left( \theta \alpha \varepsilon \right),
	\end{multline*}
	where $\tilde{\varphi}_i(x)$  is the pdf of $\min_{j < i} \tilde{Y}_{t_j}^{(k)}$, which is bounded according to Lemma \ref{density_estimate}. Thus,
	\begin{equation}
		\label{hat_L_bound_left}
		\lim_{N \to \infty} \mathbb{E}[\mathbbm{1}_{\tilde{A}^{(k)}_j }] \ge1 - \tilde{L}_{t_i}.
	\end{equation}
	Combining \eqref{hat_L_bound_right}, \eqref{hat_L_bound_left},  and \eqref{lemma_estimate},  we immediately get \eqref{lemma_convergence_eq}.

	Consider now the variance
	\begin{equation}
		\label{theor_mc_var}
		\mathbb{V}[\hat{L}_{t_i}^N] = \frac{1}{N^2} \left( \sum_{k= 1}^N \mathbb{V}[\mathbbm{1}_{\{ \min_{j < i} \hat{Y}_{t_j}^{(k)}  > 0 \}}] + \sum_{k \ne l} \cov \left(\mathbbm{1}_{\{ \min_{j < i} \hat{Y}_{t_j}^{(k)}  > 0 \}}, \mathbbm{1}_{\{ \min_{j < i} \hat{Y}_{t_j}^{(l)}  > 0 \}}\right)\right).
	\end{equation}
	The covariance can be estimated by
	\begin{multline}
		\label{theor_mc_cov_estimate}
		 \cov \left(\mathbbm{1}_{\{ \min_{j < i} \hat{Y}_{t_j}^{(k)}  > 0 \}}, \mathbbm{1}_{\{ \min_{j < i} \hat{Y}_{t_j}^{(l)}  > 0 \}}\right) = \\ 
		 \mathbb{E}\left[\mathbbm{1}_{\{ \min_{j < i} \hat{Y}_{t_j}^{(k)}  > 0 \}} \mathbbm{1}_{\{ \min_{j < i} \hat{Y}_{t_j}^{(l)}  > 0 \}} \right]  -\mathbb{E} \left[\mathbbm{1}_{\{ \min_{j < i} \hat{Y}_{t_j}^{(k)}  > 0 \}} \right] \mathbb{E} \left[ \mathbbm{1}_{\{ \min_{j < i} \hat{Y}_{t_j}^{(l)}  > 0 \}} \right].	 	 
	\end{multline}
	As above and because of exchangeability, for the second term of the last expression,
	\begin{equation}
		\label{cov_second}
		\mathbb{E} \left[\mathbbm{1}_{\{ \min_{j < i} \hat{Y}_{t_j}^{(k)}  > 0 \}} \right] \mathbb{E} \left[ \mathbbm{1}_{\{ \min_{j < i} \hat{Y}_{t_j}^{(l)}  > 0 \}} \right]  \xrightarrow[N \to \infty]{}  
		\left( 1 - \tilde{L}_{t_i} \right)^2.
	\end{equation}
	Now we consider the first term of \eqref{theor_mc_cov_estimate}. Similar to above, it can be estimated by
		\begin{multline}
		\label{lemma_estimate}
		 \mathbb{E}[\mathbbm{1}_{\{\min_{j < i} \tilde{Y}_{t_j}^{(k)}\ > -\alpha\min_{j < i}(\tilde{L}_{t_j} - \hat{L}_{t_j}^N)\}} \mathbbm{1}_{\{\min_{j < i} \tilde{Y}_{t_j}^{(l)}\ > -\alpha\min_{j < i}(\tilde{L}_{t_j} - \hat{L}_{t_j}^N)\}} ]  \le \\
		 \mathbb{E}\left[\mathbbm{1}_{\{ \min_{j < i} \hat{Y}_{t_j}^{(k)}  > 0 \}} \mathbbm{1}_{\{ \min_{j < i} \hat{Y}_{t_j}^{(l)}  > 0 \}} \right]    \\ 
		\le  \mathbb{E}[\mathbbm{1}_{\{\min_{j < i} \tilde{Y}_{t_j}^{(k)}\ > -\alpha\max_{j < i}(\tilde{L}_{t_j} - \hat{L}_{t_j}^N)\}} \mathbbm{1}_{\{\min_{j < i} \tilde{Y}_{t_j}^{(l)}\ > -\alpha\max_{j < i}(\tilde{L}_{t_j} - \hat{L}_{t_j}^N)\}} ] .
	\end{multline}
	Similar to  \eqref{hat_L_bound_right} and \eqref{hat_L_bound_left}, one can show that 
	\begin{align*}
	& \lim_{N \to \infty} \mathbb{E}[\mathbbm{1}_{\{\min_{j < i} \tilde{Y}_{t_j}^{(k)}\ > -\alpha\min_{j < i}(\tilde{L}_{t_j} - \hat{L}_{t_j}^N)\}} \mathbbm{1}_{\{\min_{j < i} \tilde{Y}_{t_j}^{(l)}\ > -\alpha\min_{j < i}(\tilde{L}_{t_j} - \hat{L}_{t_j}^N)\}} ]   \ge (1 - \tilde{L}_{t_i})^2, \\	&\lim_{N \to \infty} \mathbb{E}[\mathbbm{1}_{\{\min_{j < i} \tilde{Y}_{t_j}^{(k)}\ > -\alpha\max_{j < i}(\tilde{L}_{t_j} - \hat{L}_{t_j}^N)\}} \mathbbm{1}_{\{\min_{j < i} \tilde{Y}_{t_j}^{(l)}\ > -\alpha\max_{j < i}(\tilde{L}_{t_j} - \hat{L}_{t_j}^N)\}} ]   \le (1 - \tilde{L}_{t_i})^2. 
	\end{align*}
	Hence,
	\begin{equation}
		\label{cov_first}
		 \mathbb{E}\left[\mathbbm{1}_{\{ \min_{j < i} \hat{Y}_{t_j}^{(k)}  > 0 \}} \mathbbm{1}_{\{ \min_{j < i} \hat{Y}_{t_j}^{(l)}  > 0 \}} \right]  \xrightarrow[N \to \infty]{} (1 - \tilde{L}_{t_i})^2.
	\end{equation}
	Thus, combining \eqref{cov_second} and \eqref{cov_first}, we get
	\begin{equation*}
		 \cov \left(\mathbbm{1}_{\{ \min_{j < i} \hat{Y}_{t_j}^{(k)}  > 0 \}}, \mathbbm{1}_{\{ \min_{j < i} \hat{Y}_{t_j}^{(l)}  > 0 \}}\right) \xrightarrow[N \to \infty]{}  0.
	\end{equation*}
	Considering the first term of \eqref{theor_mc_var}, it is easy to see that each variance is bounded by 1. As a result, we have
	\begin{equation*}
		\mathbb{V}[\hat{L}_{t_i}^N] \le \frac{1}{N} + \max_{k \ne l} \left|\cov \left(\mathbbm{1}_{\{ \min_{j < i} \hat{Y}_{t_j}^{(k)}  > 0 \}}, \mathbbm{1}_{\{ \min_{j < i} \hat{Y}_{t_j}^{(l)}  > 0 \}} \right) \right| \xrightarrow[N \to \infty]{}  0.
	\end{equation*}
\end{proof}
\end{appendices}